\documentclass[a4paper]{article}
\usepackage{a4wide}
\usepackage{graphicx}
\usepackage{amsmath}
\usepackage{amsthm}
\usepackage{amssymb, latexsym, mathrsfs}
\usepackage{resizegather}
\usepackage{dsfont}
\usepackage{enumerate}
\usepackage{subfigure}
\usepackage{algorithm}
\usepackage{indentfirst}
\usepackage{color}
\usepackage{transparent}
\usepackage{url}

\usepackage[affil-it]{authblk}
\usepackage[colorlinks=true,citecolor=blue]{hyperref}
\theoremstyle{plain}
\newtheorem{thm}{Theorem}

\newtheorem{prop}{Proposition}

\newcommand{\ba}[1]{\begin{array}{#1}}
\newcommand{\ea}{\end{array}}

\begin{document}

     \title{\textbf{Global Asymptotic Stability for General Linear MIMO Distributed Systems: An Approach Based on Robust-Adaptive Controllers}}
\author[1]{Daniel O'Keeffe\thanks{Research is supported by the Irish Research Council enterprise partnership scheme (Award No. R16920) in collaboration with University College Cork, Ireland and United Technologies Research Centre Ireland Ltd.}\thanks{Email: \tt\small{danielokeeffe@umail.ucc.ie}; Corresponding author}}
\author[2]{Stefano Riverso\thanks{Email: \tt\small RiversS@utrc.ucc.com}}
\author[2]{Laura Albiol-Tendillo\thanks{Email: \tt\small AlbiolL@utrc.ucc.com}}
\author[1,3]{Gordon Lightbody\thanks{Email: \tt\small g.lightbody@ucc.ie}}
\affil[1]{\textnormal{Control \& Intelligent Systems Group, School of Engineering,
University College of Cork, Ireland}}
\affil[2]{\textnormal{United Technologies Research Centre Ireland Ltd, 4th Floor Penrose Business Centre, Cork, Ireland}}
\affil[3]{\textnormal{MaREI-SFI Research Centre, University College Cork, Ireland}}

     \date{\textbf{Technical Report}\\ January, 2018}

     \maketitle

     \begin{abstract}
       Stability is a critical feature of distributed linear multi-input-multi-output systems. Global asymptotic stability usually can be guaranteed when using decentralised or distributed control architectures, if: (i) conservative controllers are designed, (ii) collective stability conditions are satisfied, or (iii) interaction terms are neutral. 
This paper extends the collective stability method to incorporate adaptive controllers, and shows that this method is insufficient for systems with large-gain interconnections. Subsequently, we show that global asymptotic stability can be systematically ensured by exploiting vector Lyapunov functions and algebraic Riccati equations. This leads to a scalable distributed architecture where local controllers require information from corresponding subsystems and neighbouring controllers. Conveniently, the communication flow has the same topology as the interconnection graph. Theoretical results are validated through application of the proposed architecture to voltage control of a DC power network.  
     \end{abstract}

     \newpage

     \section{Introduction}
          Large-scale systems (LSS) that are often defined as interacting subsystems, such as communication, banking, supply chain, process \& chemical and electrical power systems, are some of the cornerstones of modern society. Over the last decade, the prominence of the Internet has led to increased demands for local and scalable algorithms that can achieve global objectives within typical LSS \cite{Anuradha2013}. Decentralised and distributed control architectures have become attractive alternatives to conventional centralised approaches due to growing complexities within LSS \cite{Lunze,Bakule1988a}. Despite this, centralised control architectures have well-known design criteria for guaranteeing performance and stability \cite{Macijejowski1989,Skogestad2007}. A key issue of LSS is to guarantee global asymptotic stability (GAS) and suitable performance of closed-loop coupled subsystems when equipped with decentralised controllers \cite{Siljak2011}. 

Existing decentralised and distributed controllers are commonly based on weak interactions between subsystems \cite{Riverso2013,Riverso2013a}. Such controllers treat interconnections as disturbances, requiring controllers to be robust with respect to neighbouring subsystems. This leads to conservative designs and consequently the retuning of local controllers, as demonstrated in \cite{OKeeffe2017a}.

         In general, GAS can be ensured for linear systems by using;
\begin{itemize}
\item Offline iterative decentralised controller tuning based on global knowledge.
\item Aggregated connective stability methods.
\item Linear matrix inequalities.
\end{itemize}

         The first method guarantees GAS by analysing the closed-loop eigenvalues of the global, interconnected system. If eigenvalues are unstable, decentralised controllers are iteratively retuned until these eigenvalues become stable. This method is centralised, meaning that it is not scalable. As overall system size grows, it becomes increasingly more difficult to determine the local controller/s which require retuning in order to maintain GAS according to offline analysis. Furthermore, this method lacks robustness to uncertainty of topology and interconnection models.

To overcome this, the connective stability approach was proposed in \cite{Lunze,Bakule1988a,Siljak2011,Djordjevic87} for unconstrained LSS. This uses an aggregate of subsystem and interconnection models to provide sufficient conditions for implying GAS. The method is scalable and robust to uncertainty as analysis is performed in a distributed fashion, and upper-bounds on interconnection models are incorporated. However, as alluded to in \cite{Riverso2013a}, it is limited according to the geometric small-gain theorem, which requires the interconnection gains to be sufficiently small or i.e. subsystem interactions are weak.

To circumvent small-gain conditions, the neutrality of interaction terms can be determined if state-space interconnection models are skew-symmetric. Recent developments have used linear matrix inequalities (LMIs) to ensure interactions are neutral \cite{Riverso2013,Tucci2016c,Tucci2016a,Sadabadi2017a}. LMIs are framed within a convex optimsation problem where the state-feedback control gains are solved such that the solution to local structured Lyapunov equations render interconnection terms close to zero and maintain GAS. However, this approach is also performed offline, and hence does not consider dynamic stability during topology changes i.e. plug-and-play operations, interconnection faults.

Decentralised adaptive controllers have gained popularity in heterogeneous LSS, particularly due to their ability to handle changing dynamics and uncertain interconnections using only local information \cite{Yoo2010,Vu2017,Vu2017b}. In \cite{OKeeffe2018e}, decentralised robust-adaptive controllers were implemented in a LSS to provide robustness to uncertainty concerning system topology, dynamics and plug-and-play operations. GAS was guaranteed using the first method described above. Though the synthesis of the decentralised controllers is scalable, offline GAS analysis is not. Consequently, this paper describes a scalable distributed control architecture that can guarantee sufficient GAS conditions for multi-input-multi-output (MIMO) LSS with large-gain interconnections.

  In this context, the paper provides several contributions: (i) an extension to the conventional collective stability method by incorporating decentralised model reference adaptive controllers (MRAC); (ii) the paper demonstrates that the extended collective stability method does not satisfy stability conditions in the presence of large-gain interconnections through application of a DC power network; (iii) a distributed robust-MRAC is outlined guarantee GAS of LSS with large-gain interconnections by solving local algebraic Riccati equations (ARE), and is validated through application of distributed $\mathcal{L}_1$ adaptive controllers to the same DC power network .  

Further to (iii), each local ARE is formed by including upper-bounded interconnection models to the robust-MRAC architecture. As a result, the robust-MRAC requires the state estimates of neighbouring MRACs. However, as this is a low-bandwidth, peer-peer communication constraint, the communication flow has the same topology as the interconnection graph, and thus scalability is maintained. Ultimately, by including upper-bounded interconnection models, for arbitrary system topology reconfiguration, stability can be maintained i.e. operation does not need to be stopped in order to perform offline iterative GAS analysis as before. 

The paper is organised as follows. The definition of a general MIMO distributed LSS, equipped with decentralised state-feedback baseline controllers and decentralised MRACs is provided in section 2. The extended connective stability method is derived in section 3. In section 4, the proposed distributed robust-MRAC architecture is derived and GAS of the large-gain interconnected LSS is proven. In section 5, the DC power network of \cite{Tucci2016c,OKeeffe2018e} is used to demonstrate that the small-gain theorem is not satisfied and consequently the connective stability method is insufficient for large-gain interconnections. Finally, stability of the same DC power network is demonstrated using the proposed distributed robust-MRAC architecture. 

A version of this work has been submitted to IEEE Transactions on Automatic Control.

     \section{System Definition}
         A general large-scale linear MIMO system, is considered in state-space form as,
\begin{equation}
\mathbf{\Sigma}:
\begin{cases}
\dot{\textbf{x}}(t)= \textbf{Ax}(t)+\textbf{Bu}(t) + \textbf{Ed}(t) \\
 \textbf{y}(t) = \textbf{Cx}(t)+\textbf{Du}(t)
 \end{cases}
 \label{eqn:LSSMIMO}
\end{equation}
where $\textbf{x} \in \mathbb{R}^n$, $\textbf{u} \in \mathbb{R}^m$, $\textbf{d} \in \mathbb{R}^r$, $\textbf{y} \in \mathbb{R}^q$ are the state, input, disturbance and output respectively. The notation denoting time, i.e. $\textbf{x}(t), \textbf{u}(t)$ etc. will be used only when necessary. As (\ref{eqn:LSSMIMO}) consists of $M$ interconnected subsystems, the state is partitioned into $M$ state vectors $x_{[i]} \in \mathbb{R}^{n_i}, i \in \mathcal{M}=\{1,...,M\}$ such that $\textbf{x}=(x_{[1]},...,x_{[M]})$ and $n = \Sigma_{i\in\mathcal{M}}n_i$. The input is partitioned into $M$ vectors $u_{[i]} \in \mathbb{R}^{m_i}, i \in \mathcal{M}$ such that $\textbf{u} = (u_{[1]},...,u_{[M]})$ and $m = \Sigma_{i\in\mathcal{M}}m_i$. The disturbance and output vectors are partitioned similarly, where $d_{[i]} \in \mathbb{R}^{r_i}$ and $y_{[i]} \in \mathbb{R}^{q_i}$. The matrices of (\ref{eqn:LSSMIMO}) are defined in section \ref{sec:7.1}. \vspace{3mm} \newline
\textbf{Assumption 1.} \textit{Though $\mathbf{A}$ can be unknown, $\mathbf{B}$ is assumed to be known.}
\vspace{3mm} \newline
The dynamics of the $i^{th}$ subsystem is given by,
\begin{equation}
   \Sigma_{[i]}^{\textrm{SS}}:
   \begin{cases}
   \dot{x}_{[i]}(t) = A_{ii}x_{[i]}(t)+ B_{i}u_{[i]}^{bl+ad}(t) + E_{i}d_{[i]}(t) + \zeta_{[i]}(t) \\
   y_{[i]}(t) = C_i x_{[i]}(t) + D_iu_{[i]}^{bl+ad}(t)
   \end{cases}
   \label{eq:DGUSSCL}
   \end{equation}
where $u_{[i]}^{bl+ad} \in \mathbb{R}^{m_i}$ is the control input that consists of the baseline and adaptive control signals, and $\zeta_{[i]} = \sum_{j\in\mathcal{N}_i}A_{ij}x_{[j]}\in \mathbb{R}^{n_i}, i, j \in \mathcal{M}$ is the interconnection model, and $\mathcal{N}_i$ is the set of neighbours to subsystem $i$ defined as $\mathcal{N}_i = \{j\in\mathcal{M}:A_{ij}\neq0,i\neq j\}$. $A_{ii} \in \mathbb{R}^{n_i\textrm{x}n_i}$ is the state matrix; $\hat{A}_{ij} \in \mathbb{R}^{n_i\textrm{x}n_j}$ is the interconnection matrix, $B_{i} =  $diag$(B_1,...,B_M)\in \mathbb{R}^{n_i\textrm{x}m_i}$ since subsystems are input decoupled, $C_i \in \mathbb{R}^{q_i\textrm{x}n_i}$, $D_i \in \mathbb{R}^{q_i\textrm{x}m_i}$  and $E_i \in \mathbb{R}^{n_i\textrm{x}r_i}$ are the constant input, output, direct transmission and input disturbance matrices respectively.
\vspace{3mm} \newline
\textbf{Assumption 2.} \textit{Matrices $A_{ii}$ and $B_{i}$ have full-rank, and although $A_{ii}$ is unknown, $(A_{ii}, B_{i})$ is controllable.}
\vspace{3mm} \newline
          In order to track a reference input, $r_{[i]}(t) \in \mathbb{R}^{q_i}$, in the presence of constant exogenous disturbances, integral states between the references and outputs are added to the local subsystem model. The dynamics are defined as,
     \begin{equation}
     \xi_{[i]}(t) = \int_{0}^{t}(r_{[i]}(t) - y_{[i]}(t)) dt  = \int_{0}^{t}(r_{[i]}(t) - C_{i}x_{[i]}(t)) dt
     \end{equation}
   With this, (\ref{eq:DGUSSCL}) can be rewritten as,
    \begin{equation}
       \bar{\Sigma}_{[i]}^{\textrm{SS}}:
       \begin{cases}
       \dot{\bar{x}}_{[i]} = \bar{A}_{ii}\bar{x}_{[i]}+ \bar{B}_{i}\bar{u}_{[i]}^{bl+ad} + \bar{E}_{i}\bar{d}_{[i]} + \bar{\zeta}_{[i]} \\
       \bar{y}_{[i]} = \bar{C}_i \bar{x}_{[i]} + \bar{D}_i\bar{u}_{[i]}^{bl+ad}
       \end{cases}
       \label{eq:DGUSSCL2}
       \end{equation}
        Subsequently, the open-loop model augmented with the integral state $\xi_{[i]}$ becomes $n_i+q_i$ order, where $q_i$ is the number of controlled outputs and hence number of integral states,  $\bar{x}_{[i]} = [[x_{[i]}]^T, \xi_{[i]}]^T \in \mathbb{R}^{n_i+q_i}, \bar{u}_{[i]}^{bl+ad} \in \mathbb{R}^{m_i+q_i}$, $\bar{d}_{[i]} = [[d_{[i]}]^T, r_{[i]}]^T \in \mathbb{R}^{r_i+q_i}$, $\bar{\zeta}_{[i]} = \sum_{j\in\mathcal{N}_i}\bar{A}_{ij}\bar{x}_{[j]}\in \mathbb{R}^{n_i+q_i}$. The matrices of (\ref{eq:DGUSSCL2}) are defined as,  
            \begin{equation*}
             \begin{aligned}
             \bar{A}_{ii}
              =
             \left[ \begin{array}{cc}
             A_{ii} & 0_{n_i\textrm{x}q_i} \\
             -C_{i} & 0_{q_i\textrm{x}q_i}
             \end{array} \right]
             \bar{A}_{ij} =
               \left[ \begin{array}{cc}
                A_{ij} & 0_{n_i\textrm{x}q_j} \\
                0_{q_i\textrm{x}n_j} & 0_{q_i\textrm{x}q_j}
                \end{array} \right]
                \bar{B}_{i}
                   =
                  \left[ \begin{array}{cc}
                  B_{i} \\
                  0_{q_i\textrm{x}m_i} 
                  \end{array} \right]
           \bar{C}_i = 
              \left[ \begin{array}{ccc}
              C_i & 0_{q_i\textrm{x}q_i}\\
              0_{q_i\textrm{x}n_i}&\mathbb{I}_{q_i\textrm{x}q_i} \\
               \end{array} \right]\\
               \bar{D}_{i} = 
                            \left[ \begin{array}{cc}
                            D_{i} \\ 0_{q_i\textrm{x}m_i}
                            \end{array} \right]
             \bar{E}_{i} = 
             \left[ \begin{array}{cc}
             E_{i} & 0_{n_i\textrm{x}q_i} \\
             0_{r_i\textrm{x}r_i} & \mathbb{I}_{q_i\textrm{x}q_i}
             \end{array} \right] 
             \end{aligned}
             \end{equation*} 
           where $\bar{A}_{ii} \in \mathbb{R}^{(n_i+q_i)\textrm{x}(n_i+q_i)}$, $\bar{A}_{ij} \in \mathbb{R}^{(n_i+q_i)\textrm{x}(n_j+q_j)}$, $\bar{B}_{i}\in \mathbb{R}^{(n_i+q_i)\textrm{x}m_i}$, $\bar{C_i} \in \mathbb{R}^{2q_i\textrm{x}(n_i+q_i)}$, $\bar{D_i} \in \mathbb{R}^{2q_i\textrm{x}m_i}$ and $\bar{E}_i \in \mathbb{R}^{(n_i+r_i)\textrm{x}(r_i+q_i)}$.
   
   The overall control signal consists of the summation between the state-feedback baseline controller and MRAC, defined as,
        \begin{equation}
        \begin{aligned}
         \mathcal{C}_{[i]} : u_{[i]}^{bl+ad}(t) = u_{[i]}^{bl}(t) + u_{[i]}^{mrac}(t) \\ u_{[i]}^{bl}(t) = -K_{i}^{bl}\bar{x}_{[i]}(t)
        \label{eq:SFBL}
        \end{aligned}
        \end{equation}
        where $K_{i}^{bl} = [K_{i}^{x}, K_{i}^{\xi}] \in \mathbb{R}^{(m_i\textrm{x}n_i)+q_i}$ is the state-feedback control gain vector, $K_{i}^{x}\in \mathbb{R}^{m_i\textrm{x}n_i}$ is the proportional gain vector for the original states of (\ref{eq:DGUSSCL}), and $ K_{i}^{\xi}\in \mathbb{R}^{q_i}$ is the integral gain.
        \vspace{3mm} \newline
        \textbf{Remark 1.} \textit{The method in this paper can readily be adapted to consider purely adaptive controllers.}
        \vspace{3mm} \newline
        This work considers unmatched uncertainty, whereby uncertainty enters each subsystem through the same channel as the control input.
        To reflect this, (\ref{eq:DGUSSCL2}) is rewritten as,
        \begin{equation}
              \bar{\Sigma}_{[i]}^{\textrm{SS}}:
              \begin{cases}
              \dot{\bar{x}}_{[i]} = \hat{A}_{m}\bar{x}_{[i]}+ \bar{B}_{i}(u_{[i]}^{bl+ad} +\bar{\theta}_{[i]}^T\bar{x}_{[i]})+F\bar{E}_i\bar{d}_{[i]} + \bar{\zeta}_{[i]} \\
              \bar{y}_{[i]} = \bar{C}_i \bar{x}_{[i]} + \bar{D}_iu_{[i]}^{bl+ad}
              \end{cases}
              \label{eq:DGUSSCL3}
              \end{equation}
              where $\bar{\theta}_{[i]}\in \mathbb{R}^{(n_i+q_i)\textrm{x}m_i}$ is the unknown uncertainty vector.
\vspace{3mm} \newline
\textbf{Assumption 3.} \textit{$(\bar{A}_{ii}, \bar{B_i})$ is controllable. That is, there always exists a vector $\bar{\theta}_{[i]}$ such that, the closed-loop eigenvalues can be placed anywhere, i.e. }$(\bar{A}_{ii} - \hat{A}_{m}) = \bar{B}_i\bar{\theta}_{[i]}^T \in \mathbb{R}^{(n_i+q_i)\textrm{x}(n_i+q_i)}$, \textit{where $\hat{A}_{m}$ is Hurwitz and represents the desired closed-loop dynamics.}
\vspace{3mm} \newline 
               Moreover, as the baseline controller compensates the exogenous disturbance it can be excluded, hence the term $F = \left[ \begin{array}{cc}
                           0_{n_i\textrm{x}r_i} & 0_{n_i\textrm{x}q_i} \\
                           0_{r_i\textrm{x}r_i} & \mathbb{I}_{q_i\textrm{x}q_i}
                           \end{array} \right] $ is introduced.
\vspace{3mm}

A typical MRAC architecture consists of a reference model/state-predictor and an adjustment mechanism. The reference model/state-predictor generates an estimate of each subsystem's uncertainty using an adaptation law. Subsequently, this is used to drive the subsystem to converge towards desired dynamics. The reference model/state-predictor is defined as,
\begin{equation}
 \hat{\Sigma}_{[i]}^{\textrm{SP}}:
 \begin{cases}
 \dot{\hat{x}}_{[i]} = \hat{A}_{m}\hat{x}_{[i]}+\bar{B}_i(u_{[i]}^{mrac}+\hat{\theta}_{[i]}^T\bar{x}_{[i]})+\bar{F}\bar{E}_i\bar{d}_{[i]}\\
 \hat{y}_{[i]} = \bar{C}_{i}\hat{x}_{[i]}
 \end{cases}
 \label{eq:L1SPSS2}
 \end{equation}
where $\hat{x}_{[i]} \in \mathbb{R}^{n_i+q_i}$ is the predicted state vector; $\hat{\theta} \in \mathbb{R}^{(n_i+q_i)\textrm{x}m_i}$ is the uncertainty estimate vector.

From the perspective of the MRAC, the baseline dynamics are combined with the open-loop subsystem dynamics to form an augmented closed-loop model. Without loss of generality, the reference model/state-predictor formulation is proposed for all subsystems, 
\begin{equation}
    \hat{A}_{m} = \left[ \begin{array}{cc}
    A_{m}-B_{i}K_{m}^{x} & B_{i}K_{m}^{\xi}\\ -C_{i} & 0 
    \end{array} \right]
    \end{equation}
    where $K_{m}^{bl} = [K_{m}^{x}, K_{m}^{\xi}]$ are the state-feedback and integral gains for nominal subsystem dynamics.
    
     The MRAC control signal can be defined as, 
    \begin{equation}
    u_{[i]}^{mrac}(t) = -\hat{\theta}_{[i]}^T(t)\bar{x}_{[i]}(t)
    \label{eqn:adLaw}
    \end{equation}
    Subsequently, the update law for estimating $\hat{\theta}_{[i]}$ is derived from Lyapunov theory; in section III when applying the extended connective stability method, and in section IV when applying the proposed distributed architecture.

In steady-state, when plant dynamics have converged to the dynamics of the MRAC, the overall closed-loop model of the uncertain large-scale MIMO system can be written as,
\begin{equation}
 \mathbf{\Sigma}:
\begin{cases}
\dot{\hat{\textbf{x}}}(t)= \hat{\textbf{A}}\hat{\textbf{x}}(t)+ \bar{\textbf{E}}\bar{\textbf{d}}(t) \\
 \hat{\textbf{y}}(t) = \hat{\textbf{C}}\hat{\textbf{x}}(t)+\bar{\textbf{D}}\bar{\textbf{u}}(t)
 \end{cases}
 \label{eqn:LSSMIMOcl}
\end{equation}
where \begin{math}\hat{\textbf{A}} = \hat{\textbf{A}}_\textbf{D}+\hat{\textbf{A}}_\textbf{C} \in \mathbb{R}^{(n+q)\textrm{x}(n+q)}, \hat{\textbf{A}}_\textbf{D}\end{math} =  diag\begin{math}(\hat{A}_{m1},...,\hat{A}_{mM})\end{math} contains the decoupled desired dynamics and \begin{math}\hat{\textbf{A}}_\textbf{C} = \hat{\textbf{A}}-\hat{\textbf{A}}_\textbf{D}\end{math} contains the off-diagonal interconnection terms.

          \section{Connective Stability}
          The idea of connective stability is to verify the local asymptotic stability of subsystems, and subsequently infer offline global stability by constructing an aggregated model that describes both local dynamics and the mutual interactions with neighbouring subsystems \cite{Lunze,Bakule1988a,Siljak2011}.

Initially, stability is proven for the decoupled adaptive subsystem. Subsequently, lower and upper-bounds, which account for interconnections, are derived in order to infer GAS from Lyapunov theory. The decoupled adaptive subsystem is defined as.
     \begin{equation}
     \begin{aligned}
        \bar{\Sigma}_{[i]}^{\textrm{SS}}:
        \dot{\hat{x}}_{[i]} = \hat{A}_{m}\hat{x}_{[i]} + \bar{B}_i(u_{[i]}^{mrac} +\hat{\theta}_{[i]}^T\bar{x}_{[i]})+ F\bar{E}_i\bar{d}_{[i]}
        \label{eq:DGUSSdec}
        \end{aligned}
        \end{equation}
         The tracking error dynamics between the decoupled subsystem and reference model/state predictor is obtained by subtracting (\ref{eq:L1SPSS2}) from (\ref{eq:DGUSSdec}), which yields,
                 \begin{equation}
                 \begin{aligned}
                 \dot{e}_{[i]}(t) = \dot{\bar{x}}_{[i]}(t) - \dot{\hat{x}}_{[i]}(t) \\
                        \dot{e}_{[i]}(t) = \hat{A}_{m}e_{[i]}(t)+\bar{B}_i\tilde{\theta}_{[i]}^T(t)\bar{x}_{[i]}(t)
                         \end{aligned}
                         \end{equation}
             where $\tilde{\theta}_{[i]}(t) = \bar{\theta}_{[i]}(t)-\hat{\theta}_{[i]}(t) \in \mathbb{R}^{(n_i+q_i)\textrm{x}m_i}$ is the uncertainty estimate error. The adaptive law in (\ref{eqn:adLaw}) is derived from Lyapunov theory by ensuring local adaptive subsystems are asymptotically stable in the presence of no interconnections.     
           \begin{prop}Decoupled adaptive subsystems of (\ref{eq:DGUSSdec}) are locally asymptotically stable for all $i = 1,...,M$, if for an arbitrary $Q_i = Q_i^T > 0 \in \mathbb{R}^{(n_i+q_i)\textrm{x}(n_i+q_i)}$, there exists $P_i = P_i^T > 0 \in \mathbb{R}^{(n_i+q_i)\textrm{x}(n_i+q_i)}$ that satisfies the Lyapunov equation $P_iA_{m}+A_{m}^TP_i< -Q_i$
           \end{prop}
                
          \begin{proof}
              It is convenient to use a non-quadratic Lyapunov function candidate for linear systems \cite{Bakule1988a}. The energy trajectories are mapped by defining the Lyapunov function in terms of tracking error and estimate error measurements,
           \begin{equation}
               \mathcal{V}_{[i]}(e_{[i]}(t),\tilde{\theta}_{[i]}(t) = \sqrt{e_{{[i]}(t)}^TP_{i}e_{[i]}(t)}+\tilde{\theta}_{[i]}^T(t)\Gamma_{i}^{-1}\tilde{\theta}_{[i]}(t)
               \label{eqn:Vi}
               \end{equation} 
                where $\Gamma_{i} \in \mathbb{R}^+$ is the adaptive gain. The derivative of (\ref{eqn:Vi}) yields,
            \begin{equation}
            \begin{aligned}
            \dot{\mathcal{V}}_{[i]}^{dec} = \frac{d\mathcal{V}_{[i]}}{de_{[i]}}\dot{e}_{[i]} + \frac{d\mathcal{V}_{[i]}}{d\theta_{[i]}}\dot{\theta}_{[i]} = \frac{1}{2} (e_{[i]}^TP_{i}e_{[i]})^{-1/2}(e_{[i]}^TP_{i}\dot{e}_{[i]} 
            \\
            + \dot{e}_{[i]}^TP_{i}e_{[i]})+2\tilde{\theta}_{[i]}^T\Gamma_{i}^{-1}\dot{\tilde{\theta}}_{[i]}
            \\
            = \frac{1}{2} (e_{[i]}^TP_{i}e_{[i]})^{-1/2}(e_{[i]}^T(\hat{A}_{m}^TP_{i} + P_{i}\hat{A}_{m})e_{[i]}
            \\
             + 2e_{[i]}^TP_{i}\bar{B}_i\tilde{\theta}_{[i]}\hat{x}_{[i]})+2\tilde{\theta}_{[i]}^T\Gamma_{i}^{-1}\dot{\tilde{\theta}}_{[i]}
            \end{aligned}
            \label{dotVi}
            \end{equation}
            Since $\dot{\tilde{\theta}}_{[i]}(t) \triangleq \dot{\bar{\theta}}_{[i]}(t) - \dot{\hat{\theta}}_{[i]}(t)$. As the unknown uncertainty vector is a constant,
 $\dot{\theta}_{[i]}(t) = 0$,
            \begin{equation}
            \begin{aligned}
            \dot{\mathcal{V}}_{[i]}^{dec} = \frac{1}{2} (e_{[i]}^TP_{i}e_{[i]})^{-1/2}e_{[i]}^T(\hat{A}_{m}^TP_{i} + P_{i}\hat{A}_{m})e_{[i]} 
            \\
            + \frac{e_{[i]}^TP_{i}\bar{B}_{i}\tilde{\theta}_{[i]}\hat{x}_{[i]}}{(e_{[i]}^TP_{i}e_{[i]})^{1/2}} - 2\tilde{\theta}_{[i]}^T\Gamma_{i}^{-1}\dot{\hat{\theta}}_{[i]}
            \end{aligned}
            \label{eqn:Vdot4}
            \end{equation}
            For (\ref{eqn:Vdot4}) to be at least negative semi-definite, the update law for the parameter estimate is selected as,
           \begin{equation}
                \dot{\hat{\theta}}_{[i]} = -\Gamma_{i}\frac{e_{[i]}^TP_{i}\bar{B}_{i}\hat{x}_{[i]}}{2(e_{[i]}^TP_{i}e_{[i]})^{1/2}}
                \label{eqn:update}
                \end{equation}
               Finally,
               \begin{equation}
               \dot{\mathcal{V}}_{[i]}^{dec} \leq -\frac{1}{2}\frac{e_{[i]}^TQ_{i}e_{[i]}}{(e_{[i]}^TP_{i}e_{[i]})^{1/2}} \leq 0
               \label{Vdec}
               \end{equation}
               \end{proof}

        The stability properties of the overall LSS can be established collectively by defining the vector Lyapunov function as,
            \begin{equation}
            \mathbf{V}(t)
             = \left[ \begin{array}{c}
            \mathcal{V}_{[1]}(t)\\
            \vdots\\
            \mathcal{V}_{[M]}(t)
            \end{array} \right]
            \end{equation} 
       \begin{thm}
       The overall aggregate model, represented by,
      \begin{equation}
              \mathbf{\dot{V}}(t) \leq \mathbf{MV}(t) + \mathbf{\Phi}
              \label{eqn:AggModel}
              \end{equation}
              is a sufficient condition for the GAS of (\ref{eqn:LSSMIMOcl}) if \begin{enumerate}
                \item $\mathbf{M}$, of equation (\ref{eqn:M}), is asymptotically stable
                \item $\mathbf{||M||_1 > \Phi}$
              \end{enumerate}
              \end{thm}
            
              \begin{proof}
    The lower and upper-bounds of (\ref{eqn:Vi}) are defined as,
        \begin{equation}
        \begin{aligned}
        \lambda_{min}(P_{i})||e_{[i]}||+ \frac{||\theta||^2}{\Gamma_{i}} \leq ||\mathcal{V}_{[i]}|| \leq \lambda_{max}(P_{i})||e_{[i]}||
        +\frac{||\theta||^2}{\Gamma_{i}}
        \end{aligned}
        \label{eqn:bounds1}
        \end{equation}
        where, $\lambda_{min}$, $\lambda_{max}$ are minimum and maximum eigenvalues, which correspond to lower and upper bounds of the matrices. From \cite{L12010}, the maximum bound on the parameter estimate is defined as, $\theta_{max} = 4\max_{\theta \in \Theta}||\theta||^{2}_1$.
       
        The derivative of the Lyapunov function for the interconnected subsystem can be written as,
        \begin{equation}
        \dot{\mathcal{V}}_{[i]} = \frac{d\mathcal{V}_{[i]}}{de_{[i]}}(\hat{A}_{m}e_{[i]} + \bar{B}_i\tilde{\theta}_{[i]}\bar{x}_{[i]} + \hat{\zeta}_{[i]}) + \frac{d\mathcal{V}_{[i]}}{d\tilde{\theta}_{[i]}}\dot{\tilde{\theta}}_{[i]}
        \end{equation}
       Using (\ref{dotVi}) and the parameter estimate (\ref{eqn:update}), this can be simplified to,
         \begin{equation}
        \dot{\mathcal{V}}_{[i]} = \dot{\mathcal{V}}_{[i]}^{dec} + \frac{d\mathcal{V}_{[i]}}{de_{[i]}}\hat{\zeta}_{[i]} 
        \end{equation}
From (\ref{Vdec}), the upper-bound of the decoupled Lyapunov derivative is,
        \begin{equation}
        ||\dot{\mathcal{V}}_{[i]}^{dec}|| \leq -\frac{\lambda_{min}(Q_{i})}{2\sqrt{\lambda_{max}(P_{i})}}||e_{[i]}|| 
        \end{equation}
        To determine the upper-bound on $\frac{d\mathcal{V}_{[i]}}{de_{[i]}}$, the first term on the right hand side of (\ref{dotVi}) is rearranged. 
        \begin{equation}
                    ||\frac{d\mathcal{V}_{[i]}}{de_{[i]}}|| = ||(e_{[i]}^TP_{i}e_{[i]})^{-1/2}e_{[i]}^TP_{i}|| \leq \frac{\lambda_{max}(P_{i})}{\sqrt{\lambda_{min}(P_{i})}}
                    \end{equation}
        The upper-bound on the interconnection gain is defined as,
            \begin{equation}
            ||\hat{\zeta}_{[i]}|| \leq \sum_{j\in\mathcal{N}_{i}}||{\hat{A}_{ij}}||.||{\hat{x}_{[j]}}||
            \label{eqn:interconnect}
            \end{equation}
    Therefore,
    \begin{equation}
    \begin{aligned}
    \dot{\mathcal{V}}_{[i]} \leq -\frac{\lambda_{min}(Q_{i})}{2\sqrt{\lambda_{max}(P_{i})}}||e_{[i]}|| +
     \frac{\lambda_{max}(P_{i})}{\sqrt{\lambda_{min}(P_{i})}}\sum_{j\in\mathcal{N}_i}||{\hat{A}_{ij}}||.||{\hat{x}_{[j]}}||
    \end{aligned}
    \label{Vdotf}
    \end{equation}
   From (\ref{Vdotf}), the state vector from neighbouring subsystem $j$ appears. It cannot be guaranteed that $\hat{x}_{[j]}$ is at least negative semi-definite, to ensure (\ref{Vdotf}) is also negative semi-definite. With regards to a regulation problem, where small-signals $\hat{x}_{[i]}$ and $\hat{x}_{[j]}$ would be regulated to zero, this seems valid. However, this is not the case for reference tracking. This problem can be overcome by making the reference model/predictor dependent on the interconnection model i.e. interconnection term now defined as $\tilde{\zeta}_{[i]}  = \sum_{j\in\mathcal{N}_i}{\hat{A}_{ij}}.{e_{[j]}} $, where $e_{[j]} = \hat{x}_{[i]} - \hat{x}_{[j]}$.
  
  Therefore, (\ref{Vdotf}) becomes, 
   \begin{equation}
       \begin{aligned}
       \dot{\mathcal{V}}_{[i]} \leq -\frac{\lambda_{min}(Q_{i})}{2\sqrt{\lambda_{max}(P_{i})}}||e_{[i]}||+        \frac{\lambda_{max}(P_{i})}{\sqrt{\lambda_{min}(P_{i})}}\sum_{j\in\mathcal{N}_i}||{\hat{A}_{ij}}||.||{e_{[j]}}||
       \end{aligned}
       \label{Vdotf2}
       \end{equation}  
    From (\ref{eqn:bounds1}), the tracking error is lower and upper-bounded by,
    \begin{equation}
    \frac{\mathcal{V}_{[i]} - \frac{\theta_{max}}{\Gamma_{i}}}{\sqrt{\lambda_{max}(P_{i})}} \leq ||e_{[i]}|| \leq \frac{\mathcal{V}_{[i]} - \frac{\theta_{max}}{\Gamma_{i}}}{\sqrt{\lambda_{min}(P_{i})}}
    \label{eqn:errorRe}
     \end{equation} 
   
   Equation (\ref{eqn:errorRe}) assumes that  initialisation tracking and parametric estimate errors are zero. For non-zero initialisation state and parametric estimation errors, from \cite{L12010}, the following upper-bound is used,
  \begin{equation} ||e_{[i]}|| \leq \rho_{[i]}
  \end{equation}
  where,
  \begin{equation}
  \rho_{[i]}\triangleq\sqrt{\frac{(\mathcal{V}_{[i]}(0)-\frac{\theta_{max}}{\Gamma_{i}})e^{-\alpha_i t}}{\lambda_{min}(P_{i})}+\frac{\theta_{max}}{\Gamma_{i}\lambda_{min}(P_{i})}}
  \end{equation}
  and,
  \begin{equation}
  \alpha \triangleq \frac{\lambda_{min}(Q_{i})}{\lambda_{max}(P_{i})}
  \end{equation}
  
    Rearranging (\ref{Vdotf2}) by replacing $e_{[i]}$ with its lower-bound and $e_{[j]}$ with its upper-bound yields,
    \begin{equation}
    \begin{aligned}
    \dot{\mathcal{V}}_{[i]} \leq -\frac{\lambda_{min}(Q_{i})}{2\lambda_{max}(P_{i})}\mathcal{V}_{[i]} +  \frac{\lambda_{max}(P_{i})}{\sqrt{\lambda_{min}(P_{i})\lambda_{min}(P_{j})}}\sum_{j\in\mathcal{N}_i}||{\hat{A}_{ij}}||\mathcal{V}_{[j]} +  \frac{\lambda_{min}(Q_{i})}{2\lambda_{max}(P_{i})}\frac{\theta_{max}}{\Gamma_{i}}\\-\frac{\lambda_{max}(P_{i})}{\sqrt{\lambda_{min}(P_{i})\lambda_{min}(P_{j})}}\frac{\theta_{max}}{\Gamma_{j}}\sum_{i\in\mathcal{N}_j}||\hat{A}_{ji}||
    \end{aligned}
    \label{Vdotf3}
    \end{equation}
Rearranging (\ref{Vdotf3}) in aggregated form yields (\ref{eqn:AggModel}), where
\begin{gather}
{\mathbf{M} =
                                                                                                        \left[\begin{array}{cccc}
                                                                                                        -\frac{\lambda_{min}(Q_{[1]})}{2\lambda_{max}(P_{[1]})} & \frac{\lambda_{max}(P_{[1]})}{\sqrt{\lambda_{min}(P_{[1]})\lambda_{min}(P_{[2]})}}||\hat{A}_{12}|| & \cdots & \frac{\lambda_{max}(P_{[1]})}{\sqrt{\lambda_{min}(P_{[1]})\lambda_{min}(P_{[N]})}}||\hat{A}_{1N}|| \\
                                                                                                        \frac{\lambda_{max}(P_{[2]})}{\sqrt{\lambda_{min}(P_{[2]})\lambda_{min}(P_{[1]})}}||\hat{A}_{21}|| & -\frac{\lambda_{min}(Q_{[2]})}{2\lambda_{min}(P_{[2]})} & \cdots & \frac{\lambda_{max}(P_{[2]})}{\sqrt{\lambda_{min}(P_{[2]})\lambda_{max}(P_{[N]})}}||\hat{A}_{2N}|| \\
                                                                                                        \vdots & \vdots & \ddots & \vdots \\
                                                                                                        \frac{\lambda_{max}(P_{[N]})}{\sqrt{\lambda_{min}(P_{[N]})\lambda_{min}(P_{[1]})}}||\hat{A}_{N1}|| & \frac{\lambda_{max}(P_{[N]})}{\sqrt{\lambda_{min}(P_{[N]})\lambda_{min}(P_{[2]})}}||\hat{A}_{N2}|| & \cdots & -\frac{\lambda_{min}(Q_{[N]})}{2\lambda_{max}(P_{[N]})}\\    
                                                                                                        \end{array} \right]}
                                                                                                        \label{eqn:M}                                            
\end{gather}
and,
\begin{gather}
\mathbf{\Phi} =       
                                                                                                         \left[ \begin{array}{cccc}
                                                                                                                            \theta_{max}(\frac{\lambda_{min}(Q_{[1]})}{2\Gamma_{1}\lambda_{max}(P_{[1]})}-\frac{\lambda_{max}(P_{[1]})||\hat{A}_{21}||}{\Gamma_{2}\sqrt{\lambda_{min}(P_{[1]})\lambda_{min}(P_{[2]})}}) & \cdots &0\\
                                                                                                                            \vdots & \ddots & \vdots\\
                                                                                                                            0 & \cdots & \theta_{max}(\frac{\lambda_{min}(Q_{[N]})}{2\Gamma_{N}\lambda_{max}(P_{[N]})}-\frac{\lambda_{max}(P_{[N]})||\hat{A}_{ N-1.N]
                                                                                                                            }||}{\Gamma_{N-1}\sqrt{\lambda_{min}(P_{[N]})\lambda_{min}(P_{[N-1]})}}) \\
                                                                                                                            \end{array} \right]
                                                                                                                            \label{eqn:Phi}
                                                                                                                            \end{gather}
                                                                     
 GAS is guaranteed in a distributed manner for the overall system if, (1) the diagonal elements of $\mathbf{M}$, which are associated with each decoupled subsystem, are greater than the off-diagonal elements, which are associated with the bounded interconnections i.e.
\begin{equation}
                \frac{\lambda_{min}(Q_{i})}{2\lambda_{max}(P_{i})} > \sum_{j = 1}^{N} \frac{\lambda_{max}(P_{i})}{\sqrt{\lambda_{min}(P_{i}})}\frac{||{\hat{A}_{ij}}||}{\sqrt{\lambda_{min}(P_{j})}}
                \label{eqn:cond1}
                \end{equation}
and, (2)
\begin{equation}
                \mathbf{||M||_1 > \Phi}
                \end{equation}
As $\mathbf{\Phi}$ is block diagonal, adaptation influences the stability of the local subsystem only and does not affect neighbouring subsystems. The terms of $\mathbf{\Phi}$ can be made arbitrarily small by using  large adaptive gains.
\end{proof}

As the desired dynamics of each adaptive subsystem, reflected by $A_m$, are chosen to be homogeneous, i.e. all 
subsystems converge to the same dynamics, then: $\lambda_{max}(P_{i}) = \lambda_{max}(P_{j})$, $\lambda_{min}(P_{i}) = \lambda_{min}(P_{j})$, $\lambda_{max}(Q_{i}) = \lambda_{max}(Q_{j})$, $\lambda_{min}(Q_{i}) = \lambda_{min}(Q_{j})$, $\hat{A}_{ij} = \hat{A}_{ji}$, and $\Gamma_{i}=\Gamma_{j}$. With this, (\ref{eqn:AggModel}) is negative-definite if, 
\begin{equation}
                \frac{\lambda_{min}(Q_{i})}{2\lambda_{max}(P_{i})} > \frac{\lambda_{max}(P_{i})}{\lambda_{min}(P_{i})}N_{i}||{\hat{A}_{ij}}||
                \label{eqn:cond2}
                \end{equation}
where $N_i$ is the number of interconnections.
\vspace{3mm} \newline
\textbf{Remark 1.} \textit{On close inspection, the normalising effect of the denominator in (\ref{eqn:update}) means that, $\lim\limits_{t\rightarrow \infty}\frac{e_{[i]}^TQ_{i}e_{[i]}}{(e_{[i]}^TP_{i}e_{[i]})^{1/2}} \not= 0$ and, $\lim\limits_{t\rightarrow \infty}\dot{\hat{\theta}}_{[i]} \not= 0$. Consequently, adaptation never stops. Thus, it cannot be proven using a non-quadratic Lyapunov function candidate, as required by the conventional connective stability method, that $\mathcal{V}_{[i]}(e_{[i]}, \tilde{\theta}_{[i]})$ converges and whether the system is globally asymptotically stable.}
\vspace{1.5mm} \newline
Ultimately, the connective stability does not provide a convergent solution when adaptive controllers are employed method is restricted to small-gain interconnections, as demonstrated in section V-A, and does not provide a convergent solution when adaptive controllers are employed. 
\vspace{3mm} \newline
\textbf{Remark 2.} \textit{The concept of the "geometric small-gain theorem" applied to interconnections is discussed in \cite{Riverso2013a} with regard to tube-based decentralised model predictive controllers. State, state-error and interconnection dynamics are defined for $\bar{\Sigma}_{i}^{\textrm{SS}}$ within invariant polytope sets $x_{[i]} \in \mathbb{X}$, $z_{[i]} \in \mathbb{Z}$ and $w_{[i]} \in \mathbb{W}$ respectively. As the error dynamics are subsets of the state dynamics, it must hold that $\mathbb{Z} \supseteq \mathbb{X}$. By construction, controller gains are computed to be robust to disturbances, while as seen from equation (14), the error dynamics are a function of the coupling disturbance term. Therefore $\mathbb{W} \subseteq \mathbb{Z}$, while interconnections must also be small. In our case, $\zeta_{[i]} \equiv w_{[i]} \in A_{ij}\mathbb{X}_{[j]}$ and $\zeta_{[j]} \in A_{ji}\mathbb{X}_{[i]}$. Constructing the invariant sets, $\mathbb{Z}_{[i]} \supset A_{ij}\mathbb{X}_{[j]}$, $\mathbb{Z}_{[j]} \supset A_{ji}\mathbb{X}_{[i]}$, and since $\mathbb{X}_{[j]} \supset \mathbb{Z}_{[j]}$, the following constraint must be satisfied; $ \mathbb{X}_{[i]} \supset A_{ij}A_{ji}\mathbb{X}_{[i]}$. This requires the interconnection gain $A_{ij}A_{ji}$ to be sufficiently small.}

    \section{Conditions for global asymptotic stability via local algebraic Riccati equations}
          \label{sec:plugplay}
          This section derives a systematic approach to guaranteeing GAS, following \cite{BABANRAO2004,Pagilla2007}. Subsequent stability analysis  provides sufficient conditions for global convergence by solving a local ARE. This avoids the limiting small-gain interconnection conditions from section III.
\vspace{3mm} \newline
\textbf{Remark 3.} \textit{The MRAC now requires the state estimate measurement of neighbouring MRACs, $\hat{x}_{[j]}$. Section \ref{sec:7.2} demonstrates how it is difficult to determine GAS using this approach without the measurement of $\tilde{x}_{[j]}$. Hence, the decentralised MRAC architecture becomes a distributed architecture.}\vspace{1mm}\newline
The desired subsystem dynamics as defined by the reference model/state-predictor is rewritten as,
   \begin{equation}
   \begin{aligned}
    \hat{\Sigma}_{[i]}^{\textrm{SP}}:
    \begin{cases}
    \dot{\hat{x}}_{[i]} = \hat{A}_{m}\bar{x}_{[i]}+\bar{B}_i(u_{[i]}^{mrac}+\hat{\theta}_{[i]}^T\hat{x}_{[i]})+F\bar{E}_i\bar{d}_{[i]} + \hat{\zeta}_{[i]} \\
    \hat{y}_{[i]} = \hat{C}_i\hat{x}_{[i]}
    \end{cases}
    \label{eq:L1SPSS3}
    \end{aligned}
    \end{equation}
 The state-prediction error dynamics can be written as,
   \begin{equation}
   \begin{aligned}
   \dot{\tilde{x}}_{[i]}(t) = \dot{\bar{x}}_{[i]}(t) - \dot{\hat{x}}_{[i]}(t)   \\
   \dot{\tilde{x}}_{[i]}(t) = \hat{A}_{m}\tilde{x}_{[i]}(t)+\bar{B}\tilde{\theta}_{[i]}^T(t)\bar{x}_{[i]}(t)+\sum_{j\in\mathcal{N}_i}\hat{A}_{ij}\tilde{x}_{[j]}(t)
   \label{eqn:ErrorDyn2}
   \end{aligned}
   \end{equation}
   \vspace{3mm} \newline
   \textbf{Remark 4.} \textit{The prediction error dynamics that drive the adaptive control law are now driven by the prediction error vector of $\hat{\Sigma}_{[i]}^{\textrm{SP}}$ and $\hat{\Sigma}_{[j]}^{\textrm{SP}}$\footnote{From (\ref{eqn:adaptiveLaw2i}), $\tilde{x}_{[j]}(t)$ does not directly drive $\dot{\hat{\theta}}_{[i]}(t)$, but it does affect $\tilde{x}_{[i]}(t)$ and therefore $\dot{\tilde{x}}_{[i]}(t)$}.}
   \vspace{3mm} \newline
   \textbf{Proposition 2.}
   \textit{There exists a $P_i = P_i^T > 0$ that guarantees the GAS of (\ref{eqn:LSSMIMOcl}), if
    \begin{equation}
                  \gamma \triangleq \min\limits_{\omega \in R^+}\sigma_{min}(A_{m}-j\omega \mathbb{I}) > \sqrt{N_i\Xi^2}>0
                  \label{eqn:distance}
                   \end{equation}
   where, $\gamma$ represents the distance between the Hurwitz matrix $\hat{A}_m$ and an arbitrary marginally unstable matrix is greater than a term related to the number of interconnections $N_i$ and the upper bound on the interconnection model $\Xi^2$.} \newline
   \begin{proof} 
   In this section, a quadratic Lyapunov candidate function is used to avoid the normalising effect of the adaptive law that was seen in section III. 
   \begin{equation}
   \mathcal{V}_{[i]}(\tilde{x}_{[i]}(t), \tilde{\theta}_{[i]}(t) = \tilde{x}_{[i]}(t)^TP_{i}\tilde{x}_{[i]}(t) + \tilde{\theta}_{[i]}^T(t)\Gamma_{i}^{-1}\tilde{\theta}_{[i]}(t)
   \label{eqn:LyapVi}
   \end{equation}
   The derivative of (\ref{eqn:LyapVi}) can be written as,
   \begin{equation}
   \begin{aligned}
   \dot{\mathcal{V}}_{[i]}(t) = 2(\hat{A}_{m}\tilde{x}_{[i]}+\bar{B}\tilde{\theta}_{[i]}\bar{x}_{[i]}+\sum_{j\in\mathcal{N}_i}\hat{A}_{ij}\tilde{x}_{[j]})P_{i}\tilde{x}_{[i]} + 2\tilde{\theta}_{[i]}\Gamma_{i}^{-1}\dot{\tilde{\theta}}_{[i]}
   \label{Vi2_0}
   \end{aligned}
   \end{equation}
   which equates to,
   \begin{equation}
      \begin{aligned}
      \dot{\mathcal{V}}_{[i]} = \tilde{x}_{[i]}^T(\hat{A}_{m}^TP_{i} +P_{i}\hat{A}_{m})\tilde{x}_{[i]} +  P_{i}\tilde{x}_{[i]}^T\sum_{j\in\mathcal{N}_i}\hat{A}_{ij}\tilde{x}_{[j]}+(\sum_{j\in\mathcal{N}_i}\hat{A}_{ij}\tilde{x}_{[j]})P_{i}\tilde{x}_{[j]}
      \label{Vi2}
      \end{aligned}
      \end{equation}
   when the update law for the parameter estimate is chosen as,
   \begin{equation}
   \dot{\hat{\theta}}_{[i]} = \Gamma_{i}Proj(\hat{\theta}_{[i]}, -\tilde{x}_{[i]}^TP_{i}\bar{B}_i\bar{x}_{[i]})
   \label{eqn:adaptiveLaw2i}
   \end{equation}
   The projection operator, described in section \ref{sec:7.4}, is used to prevent parametric drift by upper-bounding the parameter estimate \textit{a priori} i.e. $\theta_{max}$, and thus provides robust adaptation.
   
   Expanding the two summation terms of (\ref{Vi2}) yields,
   \begin{equation}
   \begin{aligned}
   \tilde{x}_{[i]}^TP_{i}(\sum_{j\in\mathcal{N}_i}\hat{A}_{ij}\tilde{x}_{[j]})+(\sum_{j\in\mathcal{N}_i}\hat{A}_{ij}\tilde{x}_{[j]})^TP_{i}\tilde{x}_{[i]} =  (\hat{A}_{i1}\tilde{x}_{[1]})^TP_{i}\tilde{x}_{[1]} + \tilde{x}_{[1]}^TP_{i}(\hat{A}_{i1}\tilde{x}_{[1]})+....\\+(\hat{A}_{iM-1}\tilde{x}_{[M-1]})P_{i}\tilde{x}_{[M-1]}^T + \tilde{x}_{[M-1]}^TP_{i}(\hat{A}_{iM-1}\tilde{x}_{[M-1]})
   \label{eq:Expand}
   \end{aligned}
   \end{equation}
   where $\bar{\Sigma}_{[i]}^{\textrm{SS}} $ can have a maximum of $M-1$ neighbours.
   To achieve decentralisation, the cross-coupled $i^{th}$ state error vector and the $j^{th}$ state vector terms must be decoupled. This is done using the inequality condition in \cite{BABANRAO2004},
   \begin{equation}
    X^TY + Y^TX \leq X^TX + Y^TY
    \label{eq:ineq}
    \end{equation}
	   where $X = P_{i}\tilde{x}_{[i]}$ and $Y = \hat{A}_{ij}\tilde{x}_{[j]}$. The following is obtained using this condition.
      \begin{equation}
      \begin{aligned}
   \tilde{x}_{[i]}^TP_{i}(\hat{A}_{ij}\tilde{x}_{[j]})+(\hat{A}_{ij}\tilde{x}_{[j]})^TP_{i}\tilde{x}_{[i]} \le 
   (\hat{A}_{ij}\tilde{x}_{[j]})^T(\hat{A}_{ij}\tilde{x}_{[j]})+(P_{i}\tilde{x}_{[i]})^T(P_{i}\tilde{x}_{[i]}) = \\ \tilde{x}_{[j]}^T(\hat{A}_{ij}^T\hat{A}_{ij})\tilde{x}_{[j]} + \tilde{x}_{[i]}^TP_{i}^2\tilde{x}_{[i]}
   \label{eq:Expand2}
   \end{aligned}
   \end{equation}
   Applying this decoupling to all terms in (\ref{eq:Expand}) yields,
   \begin{equation}
   \begin{aligned}
   \tilde{x}_{[i]}^TP_{i}(\sum_{j\in\mathcal{N}_i}\hat{A}_{ij}\tilde{x}_{[j]})+(\sum_{j\in\mathcal{N}_i}\hat{A}_{ij}x_{[j]})^TP_{i}\tilde{x}_{[i]} \le  N_i\tilde{x}_{[i]}^TP_{i}^2\tilde{x}_{[i]}+
   \sum_{j\in\mathcal{N}_i}\tilde{x}_{[j]}^T(\hat{A}_{ij}^T\hat{A}_{ij})\tilde{x}_{[j]}
   \label{eq:Expand3}
   \end{aligned}
   \end{equation}
   Therefore, the decoupled terms on the right-hand-side of (\ref{eq:Expand3}) are taken as the largest interconnection terms, and substituted back into (\ref{Vi2}). 
    \begin{equation}
    \begin{aligned}
    \dot{\mathcal{V}}_{[i]} \le \tilde{x}_{[i]}^T(\hat{A}_{m}^TP_{i} +P_{i}\hat{A}_{m})\tilde{x}_{[i]} + N_i\tilde{x}_{[i]}^TP_{i}^2\tilde{x}_{[i]}+
    \sum_{j\in\mathcal{N}_i}\tilde{x}_{[j]}^T(\hat{A}_{ij}^T\hat{A}_{ij})\tilde{x}_{[j]}
    \label{Vidot3}
    \end{aligned}
    \end{equation}
The interconnection term in (\ref{Vidot3}) can be upper bounded as,
    \begin{equation}
    \sum_{j\in\mathcal{N}_i}\tilde{x}_{[j]}^T(\hat{A}_{ij}^T\hat{A}_{ij})\tilde{x}_{[j]} \le \sum_{j\in\mathcal{N}_i}\eta_{ij}^2\tilde{x}_{[j]}^T\tilde{x}_{[j]}
    \end{equation}
   where $\eta_{ji}^2 \geq \lambda_{max}(\hat{A}_{ji}^T\hat{A}_{ji})$.
The vector Lyapunov function which describes global stability can be written as,
   \begin{equation}
   \begin{aligned}
   \dot{\mathcal{V}} \leq \sum\limits_{i = 0}^{N}\tilde{x}_{[i]}^T(\hat{A}_{m}^TP_{i}+P_{i}\hat{A}_{m}+N_iP_{i}^2)\tilde{x}_{[i]} +  \sum\limits_{i = 0}^{N}(\sum_{j\in\mathcal{N}_i}\eta_{ij}^2)\tilde{x}_{[j]}^T\tilde{x}_{[j]}
   \label{eq:Vdot4}
     \end{aligned}
   \end{equation}
   The maximum eigenvalue of  $\sum_{j\in\mathcal{N}_i}\hat{A}_{ij}^T\hat{A}_{ij}$ is defined as $\xi_{i}^2 = \sum_{j\in\mathcal{N}_i}\eta_{ji}^2$ Therefore, the last term of (\ref{eq:Vdot4}) can be written as,
   \begin{equation}
    \sum_{j\in\mathcal{N}_i}\eta_{ij}^2 = \sum\limits_{j=0,j\neq i}^{N}\eta_{ij}^2 = \sum\limits_{i=0,j\neq i}^{N}\eta_{ji}^2
     \end{equation}
      The indexes are interchangeable since the upper bound of the coupling matrix applies to all subsystems and their interconnections. The interconnection term can be written as $\xi_{i}^2\sum\limits_{i = 0}^{N}\tilde{x}_{[j]}^T\tilde{x}_{[j]}$. However, the indexes do not match here. Therefore they are also interchanged so that the summation of the final term in (\ref{eq:Vdot4}) makes sense.
      
    Parametric knowledge of the interconnection model $\hat{A}_{ij}$ can be unknown. In such cases, for conservativeness i.e. worst-case design, the upper-bound of (\ref{eqn:interconnect}) is used to account for robustness to uncertainty.
      
      Defining,
      \begin{equation}
      \begin{aligned}
      \Xi^2 \triangleq \sum_{j\in\mathcal{N}_i}\lambda_{max}(\hat{A}_{ji}^T\hat{A}_{ij})=
      \sum_{j\in\mathcal{N}_i}\lambda_{max}(\hat{A}_{ij}^2)
        \end{aligned}
      \end{equation}
      since the upper-bound on $\hat{A}_{ij}$ is the same as the upper-bound $\hat{A}_{ji}$.
Finally, (\ref{eq:Vdot4}) can be written as, 
   \begin{equation}
   \begin{aligned}
   \dot{\mathcal{V}} \leq \sum\limits_{i = 0}^{N}\tilde{x}_{[i]}^T(\hat{A}_{m}^TP_{i}+P_{i}\hat{A}_{m}+N_iP_{i}^2)\tilde{x}_{[i]} +  \Xi^2\sum\limits_{i = 0}^{N}\tilde{x}_{[i]}^T\tilde{x}_{[i]} \\
   \leq \sum\limits_{i = 0}^{N}\tilde{x}_{[i]}^T(\hat{A}_{m}^TP_{i}+P_{i}\hat{A}_{m}+N_iP_{i}^2+\Xi^2\mathbb{I}_{n_i\textrm{x}n_i})\tilde{x}_{[i]}
   \end{aligned}
   \label{eq:Vdot5}
   \end{equation}
   
   The global Lyapunov function of (\ref{eq:Vdot5}) resembles an ARE. The following two Lemma's are important in establishing sufficient conditions that ensure stability of the overall LSS for the proposed decentralised adaptive controllers.
   \vspace{3mm} \newline  
   \textbf{Lemma 1.} \textit{Considering the ARE defined as:}
   \begin{equation}
   A^TP+PA + PRP + Q = 0
   \label{eqn:ARE}
   \end{equation}
   \textit{If A is Hurwitz, $R = R^T > 0$, $Q = Q^T > 0$, and the associated Hamiltonian matrix $\mathcal{H} = \left[ \begin{array}{cc}
   A & R \\
   -Q & -A^T
   \end{array} \right]$ is hyperbolic, i.e. $\mathcal{H}$ has no eigenvalues that lie on the imaginary axis, then there exists $P = P^T > 0$ that solves the ARE of (\ref{eqn:ARE}).
   In this case, $A = \hat{A}_{m}$, $Q = \Xi^2\mathbb{I}_{n_i\textrm{x}n_i}$ and $R = N_i\mathbb{I}_{n_i\textrm{x}n_i}$.}
   \vspace{3mm} \newline  
   \textbf{Lemma 2.} \textit{
   The Hamiltonian matrix $\mathcal{H}_i = \left[ \begin{array}{cc}
   \hat{A}_{m} & N_i\mathbb{I}_{n_i\textrm{x}n_i}\\
   -\Xi^2\mathbb{I}_{n_i\textrm{x}n_i} & -\hat{A}_{m}^T
   \end{array} \right]$ is hyperbolic if and only if:}
   \begin{equation}
   \min\limits_{\omega \in R^+}\sigma_{min}(\hat{A}_{m}-j\omega \mathbb{I}_{n_i\textrm{x}n_i}) > \Xi\sqrt{N_i}>0
   \label{eq:distance1}
   \end{equation}
   \textit{Proof.}
   The proof can be found in section \ref{sec:7.3}.
   
 A linear bisection method for solving this distance and ensuring (\ref{eqn:distance}) is described in \cite{Byers1988,Aboky2002}, and is summarised in section \ref{sec:7.5}.
   
   If $\mathcal{H}_i$ is hyperbolic, there exists $\varepsilon_i > 0$ such that,
   \begin{equation}
   \mathcal{H}_i = \left[ \begin{array}{cc}
   \hat{A}_{m} & N_i\mathbb{I}_{n_i\textrm{x}n_i}\\
   -(\Xi^2+ \varepsilon_i)\mathbb{I}_{n_i\textrm{x}n_i} & -\hat{A}_{m}^T
   \end{array} \right]
   \end{equation}
   is also hyperbolic. That is, if $f_i(\xi_{i}^2) := \min\limits_{\omega \in R^+}\sigma_{min}(\hat{A}_{m}-j\omega \mathbb{I}_{n_i\textrm{x}n_i}) - \Xi\sqrt{N_i}>0$, then there exists  $\varepsilon_i > 0$ such that $f_i(\Xi^2+\varepsilon_i) > 0$.
   
   The ARE of Lemma 1 can be written as,
   \begin{equation}
   \hat{A}_{m}^TP_{i}+P_{i}\hat{A}_{m} + P_{i}N_iP_{i} + (\Xi^2+\varepsilon_i)\mathbb{I}_{n_i\textrm{x}n_i} = 0
   \end{equation}
   and is solvable with $P_{i} = P_{i}^T > 0$ if Lemma's 1 and 2 are satisfied. The overall vector Lyapunov function can be represented as,
   \begin{equation}
   \begin{aligned}
   \dot{\mathcal{V}} \leq \sum\limits_{i = 0}^{N}\tilde{x}_{[i]}^T(\hat{A}_{m}^TP_{i}+P_{i}\hat{A}_{m}+N_iP_{i}^2+(\Xi^2+\varepsilon_i)\mathbb{I}_{n_i\textrm{x}n_i})\tilde{x}_{[i]}
   \leq -\sum\limits_{i = 0}^{N}\tilde{x}_{[i]}^T\varepsilon_{[i]}\tilde{x}_{[i]}
   \end{aligned}
   \label{eq:Vdot6}
   \end{equation}
Again, this has a guaranteed solution, with a positive-definite $P_{i}$ if Lemma's 1 and 2 are satisfied. Subsequently, according to Barbalat's Lemma, the error dynamics are bounded and $\lim\limits_{t\rightarrow \infty}\tilde{x}_{[i]} = 0$.
\end{proof}

Ultimately, the distributed robust-MRAC architecture can be deployed in local subsystems of linear MIMO LSS with strong-interactions and guarantee GAS, as seen in section V-C. As with all Lyapunov stability criterion, satisfying (\ref{eqn:distance}) is a sufficient but not necessary stability condition. However, conservativeness can be relaxed by tight satisfaction of (\ref{eqn:distance}) when specifying  the closed-loop dynamics via the MRAC.

     \section{Results}
   In general, power systems are increasingly becoming decentralised and distributed \cite{Anuradha2013}. Therefore, DC power distribution systems are an effective application to evaluate the distributed control architecture. As proposed in \cite{Tucci2016c,OKeeffe2018e}, 6 heterogeneous subsystems provide power to local loads in an interconnected electrical DC power network. Each subsystem is equipped with controllers $\mathcal{C}_{[i]}, i = 1,...,6$, for voltage stabilisation, where the robust-MRAC is an $\mathcal{L}_1$ adaptive controller. Parameter values defined in Table I of \cite{OKeeffe2018e} are used.
   
   \subsection{Failure of connective stability condition with large-gain interconnections}
  The small-gain theorem is a commonly used formulation to determine input-output stability of an interconnected non-linear system \cite{Khalil2002}. 
  \begin{figure}[!htb]    
  \centering
  \includegraphics[width=9cm]{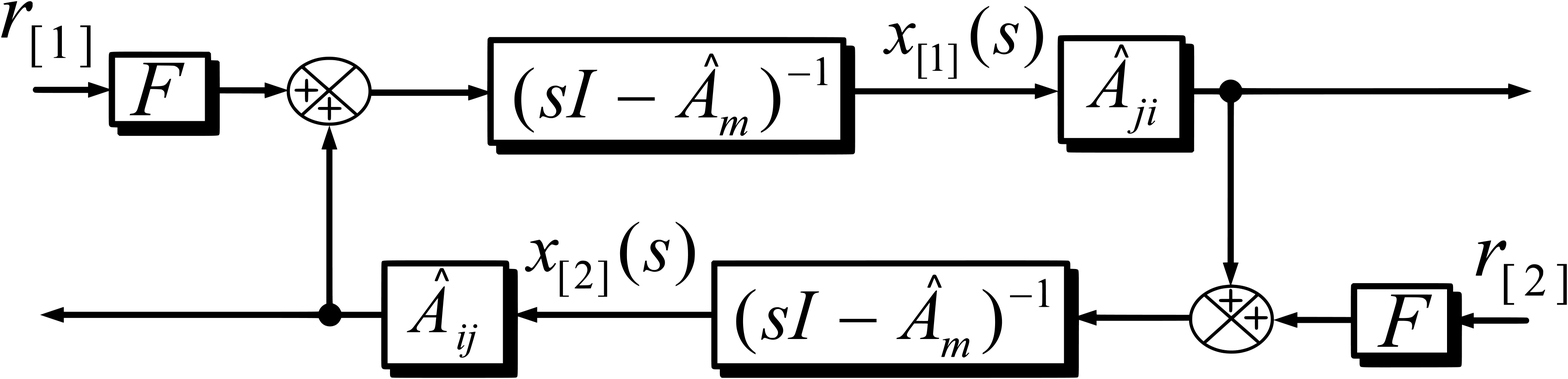}
  \caption{Feedback connection between two interconnected subsystems.}
  \label{fig:SmallGainRep}
  \end{figure}
  Using Fig.\ref{fig:SmallGainRep}, the theorem states that the feedback connection between the two interconnected subsystems\footnote{Both systems must be individually stable i.e. an unstable open-loop system must be closed-loop stable.} is input-output stable if the product of the individual infinity-norms of each converter is less than 1. That is,
 \begin{equation}
 ||\hat{A}_{12}{
 (s\mathbb{I}-\hat{A}_{m})^{-1} ||_{\infty}.||\hat{A}_{21}}(s\mathbb{I}-\hat{A}_{m})^{-1}||_{\infty} < 1
 \label{eqn:SmallGainCond}
 \end{equation}
 Note that this assumes that adaptation and convergence to desired dynamics has occurred.
 Consider the following system where subsystem $\bar{\Sigma}_{1}^{\textrm{SS}}$ is coupled with subsystem $\bar{\Sigma}_{2}^{\textrm{SS}}$. System matrices are derived from \cite{OKeeffe2018e}.
    \begin{equation*}
    \centering
    \begin{aligned}
        \hat{A}_{m}
         =
        \left[ \begin{array}{ccc}
        -3.51\textrm{x}10^6 & 4\textrm{x}10^3 & 1.13\textrm{x}10^6 \\
        5.12\textrm{x}10^6 & -9\textrm{x}10^4 & -1.65\textrm{x}10^6 \\
        0 & -1 & 0 \\
        \end{array} \right]
        \hat{A}_{12} =
                \left[ \begin{array}{ccc}
                 0 & 0 & 0\\
                 0 & 5.32\textrm{x}10^4 & 0 \\
                 0 & 0 & 0 \\
                 \end{array} \right]
            \hat{A}_{21} =
             \left[ \begin{array}{ccc}
                                    0 & 0 & 0\\
                                    0 & 3.87\textrm{x}10^4 & 0 \\
                                    0 & 0 & 0 \\
                                    \end{array} \right]
                  \\ \bar{B}_1
         =
        \left[ \begin{array}{ccc}
       1.34\textrm{x}10^7 \\
       -8.12\textrm{x}10^5 \\
        0
        \end{array} \right]
        \bar{B}_2
                 =
                \left[ \begin{array}{ccc}
                4.25\textrm{x}10^6 \\
                -5.6\textrm{x}10^5 \\
                0
                \end{array} \right]        
      \hat{C} = 
         \left[ \begin{array}{ccc}
          0 & 1 & 0 \\
          \end{array} \right]
          \end{aligned}
        \end{equation*} 
       
Consequently, the small-gain condition is,
  \begin{equation}
  (5.32\textrm{x}10^4)(3.87\textrm{x}10^4) \nless 1
  \end{equation}
which clearly is not satisfied. This two coupled subsystem example is effectively a large-gain system. Generally, this can be said for an $M$ subsystem DC power network, which results in the failure of the collective stability method for strongly-interacting subsystems .   
  
By setting $Q_{1} = \mathbb{I}_{3\textrm{x}3}$, the positive-definite matrix $P_{1}$ is solved for, 
\begin{equation*}
P_{1} =
1\textrm{x}10^4\left[ \begin{array}{ccc}
        1.4 & 0.98 & 0.25\\
        0.98 & 0.8 & 0.21 \\
        0.25 &  0.21 & 1.61 \\
        \end{array} \right] \\
\end{equation*}
where $\varepsilon_i = 1\textrm{x}10^3$,  $\lambda_{max}(P_{1}) = 2.28x\textrm{x}10^4$ and $\lambda_{min}(P_{1}) = 723.2$.

The GAS condition of (\ref{eqn:cond1}) can now be checked.
 \begin{equation}
 \frac{1}{2(2.28\textrm{x}10^4)} \ngtr \frac{2.28\textrm{x}10^4}{723.2}1.1\textrm{x}10^5
 \label{eqn:cond2}
 \end{equation}
Clearly, the left-hand-side of (\ref{eqn:cond2}) is not greater than the right-hand-side as required by (\ref{eqn:cond1}).

Ultimately, from the interconnection matrix $\hat{A}_{ij}$, the large-gain leads to large upper-bounds. One could attempt to increase $\lambda_{min}(Q_{1})$, however, as the system is linear, $\lambda_{min}(P_{1})$ and $\lambda_{max}(P_{1})$ will increase accordingly. This shows that the connective stability method is only suitable for small-gain interconnections.

\textbf{Remark 5.} \textit{Section \ref{sec:7.1} analytically shows that a typical power electronic design feature to improve local damping does not contribute to improving global stability.}

\subsection{Simulation of DC power network with distributed control architecture}

Again, the DC power network of \cite{OKeeffe2018e}
is used to demonstrate the effectiveness of distributed $\mathcal{L}_1$ adaptive controllers at guaranteeing GAS, when designed as described in section IV. 

The topology of the power network is of meshed configuration.
\begin{figure}[!htb]    
\graphicspath{ {Images/} }
\centering
\includegraphics[width=8cm]{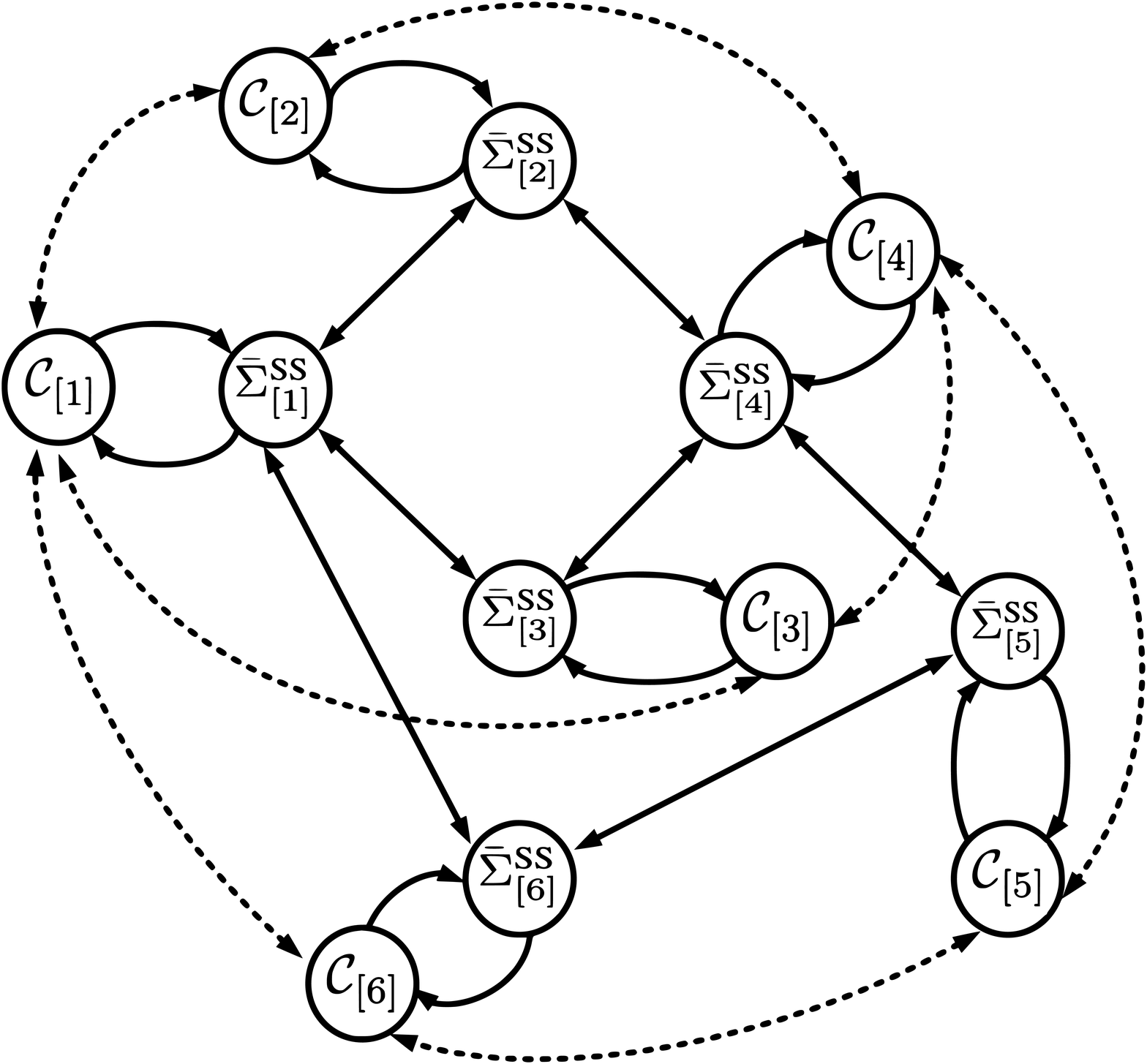}
\caption{Large-scale interconnected DC power network with low-bandwidth communications (dotted).} 
\label{fig:TestConfig}
\end{figure}
All subsystems have different output references which allows for interaction between the interconnections.

A load disturbance step of 2 kW - 3.8 kW is applied to $\bar{\Sigma}_{[2]}^{\textrm{SS}}$. 
Fig.\ref{fig:TestConfig} shows that each interconnected subsystem is affected by the load disturbance.
\begin{figure}[!htb]    
\centering
\includegraphics[width=9cm]{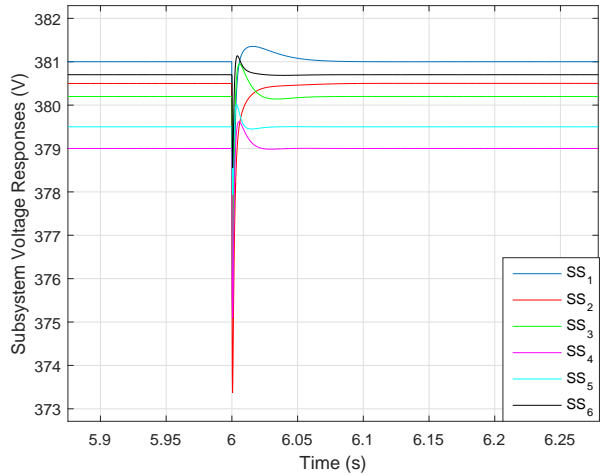}
\caption{Subsystem voltage responses for  $\bar{\Sigma}_{[2]}^{\textrm{SS}}$ load change 2 kW - 3.8 kW.}
\label{fig:DGU2LoadChange}
\end{figure}
The distributed adaptive control architecture maintains GAS and fast performance. The conservativeness of using an upper-bound on the interconnection models for all subsystems does not appear to slow the system response.
    \section{Conclusion}
    The aim of this paper was to overcome the problem of guaranteeing GAS when decentralised and distributed control architectures are employed in general linear MIMO LSS. The paper extended the conventional connective stability method by incorporating model reference adaptive control structures. Using a DC power network, the method was numerically tested. As a result, the method is shown to be insufficient in the presence of large-gain interconnections. Furthermore, the method does not ensure convergent adaptive control laws as it relies on a non-quadratic Lyapunov function. 
    
    A distributed control architecture was presented to overcome the small-gain limitation. Distributed $\mathcal{L}_1$ adaptive controllers are used to augment state-feedback baseline controllers. 
    By ensuring the Hamiltonian matrix associated with each local ARE is hyperbolic, GAS is sufficiently guaranteed when conditions dependent on the desired dynamics, number of interconnections and the upper-bound of interconnection models are satisfied. Structurally, each subsystem at most requires the measurement of neighbouring predictor errors. This caveat does not impede the scalability of the architecture as the communication topology is neighbour-to-neighbour i.e. same topology as the interconnection graph. 
    
    The distributed control architecture is tested on the same DC power network, where GAS and fast performance are maintained during a load disturbance to one of the subsystems. Future work will evaluate the proposed architecture in response to more realistic DC power system settings including; bus-connected topologies, plug-and-play operations, reference tracking, and interaction with coordination control layers.

    \section{Appendix}
    \subsection{Matrices of General Linear MIMO System Model}\label{sec:7.1}
    The matrices of (\ref{eqn:LSSMIMO}) are defined as,
    \begin{equation}
\underbrace{\left[\begin{array}{c}
\dot{x}_{[1]} \\
\dot{x}_{[2]} \\
\dot{x}_{[3]} \\
\vdots \\
\dot{x}_{[M]} 
\end{array} \right]}_{\dot{\textbf{x}}}=
\underbrace{
\left[\begin{array}{ccccc}
A_{11} & A_{12} & A_{13}  & \cdots & A_{1M} \\
A_{21} & A_{22} & A_{23}  & \cdots & A_{2M} \\ 
A_{31} & A_{32} &  A_{33}  & \cdots & A_{3M} \\
\vdots & \vdots & \vdots & \ddots & \vdots
\\
A_{M1} & A_{M2} & A_{M3} & \cdots & A_{MM}
\end{array} \right]}_{\textbf{A}}
\underbrace{\left[\begin{array}{c}
x_{[1]} \\
x_{[2]} \\
x_{[3]} \\
\vdots \\
x_{[M]} 
\end{array} \right]}_{\textbf{x}} + \underbrace{
\left[\begin{array}{ccccc}
B_{1} & 0 & 0 & \cdots  & 0 \\
0 & B_{2} & 0 & \ddots  & \vdots  \\ 
0 & 0 &  B_{3} & \ddots & \vdots \\
\vdots & \ddots & \ddots & \cdots & 0
\\
0 & \cdots & 0 & 0 & B_{M} 
\end{array} \right]}_{\textbf{B}}
\underbrace{\left[\begin{array}{ccccc}
u_{[1]} \\
u_{[2]} \\
u_{[3]} \\
\vdots \\
u_{[N]} 
\end{array} \right]}_{\textbf{u}} 
\\
+ \underbrace{
\left[\begin{array}{ccccc}
E_{1} & 0 & 0 & \cdots  & 0 \\
0 & E_{2} & 0 & \ddots  & \vdots  \\ 
0 & 0 &  E_{3} & \ddots & \vdots \\
\vdots & \ddots & \ddots & \cdots & 0
\\
0 & \cdots & 0 & 0 & E_{M} 
\end{array} \right]}_{\textbf{E}}
\underbrace{\left[\begin{array}{ccccc}
d_{[1]} \\
d_{[2]} \\
d_{[3]} \\
\vdots \\
d_{[M]} 
\end{array} \right]}_{\textbf{d}} 
\hspace{2mm}
;
\hspace{2mm}
\underbrace{\left[\begin{array}{ccccc}
y_{[1]} \\
y_{[2]} \\
y_{[3]} \\
\vdots \\
y_{[M]} 
\end{array} \right]}_{\textbf{y}} = 
\underbrace{
\left[\begin{array}{ccccc}
C_{1} & 0 & 0 & \cdots  & 0 \\
0 & C_{2} & 0 & \ddots  & \vdots  \\ 
0 & 0 &  C_{3} & \ddots & \vdots \\
\vdots & \ddots & \ddots & \cdots & 0
\\
0 & \cdots & 0 & 0 & C_{M} 
\end{array} \right]}_{\textbf{C}} 
\underbrace{\left[\begin{array}{c}
x_{[1]} \\
x_{[2]} \\
x_{[3]} \\
\vdots \\
x_{[M]} 
\end{array} \right]}_{\textbf{x}}
\\
+\underbrace{
\left[\begin{array}{ccccc}
D_{1} & 0 & 0 & \cdots  & 0 \\
0 & D_{2} & 0 & \ddots  & \vdots  \\ 
0 & 0 &  D_{3} & \ddots & \vdots \\
\vdots & \ddots & \ddots & \cdots & 0
\\
0 & \cdots & 0 & 0 & D_{M} 
\end{array} \right]}_{\textbf{D}} 
\underbrace{\left[\begin{array}{c}
u_{[1]} \\
u_{[2]} \\
u_{[3]} \\
\vdots \\
u_{[M]} 
\end{array} \right]}_{\textbf{u}} 
\end{equation}
    \subsection{Guaranteeing GAS without measurement of neighbouring prediction errors}\label{sec:7.2}
    The Lyapunov function of  (\ref{eqn:LyapVi}) is used. The decoupled state-predictor of (\ref{eq:L1SPSS2}) is also used. 
 The state-prediction error dynamics are defined as,
    \begin{equation}
    \dot{\tilde{x}}_{[i]}(t) = \hat{A}_{m}\tilde{x}_{[i]}(t)+\bar{B}_{i}\tilde{\theta}_{[i]}^T(t)\bar{x}_{[i]}(t)+\sum_{j\in\mathcal{N}_i}\hat{A}_{ij}\bar{x}_{[j]}(t)
       \label{eqn:ErrorDyn3}
       \end{equation}
    \newline
    \textbf{Remark 6.} \textit{The difference here is that the state error dynamics are directly a function of the neighbouring states $\bar{x}_{[j]}(t)$ rather than the neighbouring prediction error $\tilde{x}_{[j]}(t)$.}
    \vspace{3mm} \newline
The derivative of the local Lyapunov function is,
    \begin{equation}
       \begin{aligned}
       \dot{\mathcal{V}}_{[i]} = \tilde{x}_{[i]}^T(\hat{A}_{m}^TP_{i} +P_{i}\hat{A}_{m})\tilde{x}_{[i]} + P_{i}\tilde{x}_{[i]}^T\sum_{j\in\mathcal{N}_i}\hat{A}_{ij}x_{[j]}+(\sum_{j\in\mathcal{N}_i}\hat{A}_{ij}\tilde{x}_{[j]})P_{i}x_{[j]}
       \label{Vidot3_1}
       \end{aligned}
       \end{equation}
       The update law for the parameter estimate is (\ref{eqn:adaptiveLaw2i}), as before.
    Expanding the two summation terms of (\ref{Vidot3_1}) yields,
       \begin{equation}
       \begin{aligned}
       \tilde{x}_{[i]}^TP_{i}(\sum_{j\in\mathcal{N}_i}\hat{A}_{ij}\tilde{x}_{[j]})+(\sum_{j\in\mathcal{N}_i}\hat{A}_{ij}\tilde{x}_{[j]})^TP_{i}\tilde{x}_{[i]} =  (\hat{A}_{i1}\tilde{x}_{[1]})^TP_{i}\tilde{x}_{[1]} + \tilde{x}_{[1]}^TP_{[1]}(\hat{A}_{i1}\tilde{x}_{[1]})+....\\+(\hat{A}_{iN-1})\tilde{x}_{[N-1]}P_{[1]}\tilde{x}_{[N-1]}^T + \tilde{x}_{[N-1]}^TP_{[1]}(\hat{A}_{iN-1}\tilde{x}_{[N-1]})
       \label{eq:Expand2}
       \end{aligned}
       \end{equation}
 Recalling the inequality of (\ref{eq:ineq}) we get,
    \begin{equation*}
    \begin{aligned}
    (\hat{A}_{ij}x_{[j]})^TP_{i}\tilde{x}_{[i]} + \tilde{x}_{[i]}^TP_i(\hat{A}_{ij}x_{[j]}) \leq \tilde{x}_{[i]}^TP_{i}^2\tilde{x}_{[i]}+ x_{[j]}^T(\hat{A}_{ij}^T\hat{A}_{ij})x_{[j]}
    \end{aligned}
    \end{equation*}
    Or, 
    \begin{equation}
    \begin{aligned}
    (\sum_{j\in\mathcal{N}_i}\hat{A}_{ij}x_{[j]})^TP_{i}\tilde{x}_{[i]} + \tilde{x}_{[i]}^TP_i(\sum_{j\in\mathcal{N}_i}\hat{A}_{ij}x_{[j]}) \leq  N_i\tilde{x}_{[i]}^TP_{i}^2\tilde{x}_{[i]} + \sum_{j\in\mathcal{N}_i}x_{[j]}^T(\hat{A}_{ij}^T\hat{A}_{ij})x_{[j]}
    \end{aligned}
    \end{equation}
Following the same procedure from section III-A i.e. using  $\tilde{x}_{[i]} = \bar{x}_{[i]} - \hat{x}_{[i]}$, the equations (42)-(44) and Lemma's 1 and 2, the overall Lyapunov function can be written as,
     \begin{equation}
       \begin{aligned}
       \dot{\mathcal{V}} \leq \sum\limits_{i = 0}^{N}\bar{x}_{[i]}^T(\hat{A}_{m}^TP_{i}+P_{i}\hat{A}_{m}+N_iP_{i}^2+(\Xi^2\mathbb{I}_{n_i\textrm{x}n_i}))\bar{x}_{[i]}-
       \hat{x}^T_{[i]}(\hat{A}_{m}^TP_{i}+P_{i}\hat{A}_{m}+N_iP_{i}^2+(\Xi^2\mathbb{I}_{n_i\textrm{x}n_i}))\hat{x}_{[i]} \\\dot{\mathcal{V}}\leq
       -\sum\limits_{i = 0}^{N} \bar{x}_{[i]}^T\varepsilon_i\bar{x}_{[i]}+ \hat{x}_{[i]}^T(-(\Xi^2+\varepsilon_i)\mathbb{I}_{n_i\textrm{x}n_i})\hat{x}_{[i]}
       \leq -\sum\limits_{i = 0}^{N}\tilde{x}_{[i]}^T\varepsilon_i\tilde{x}_{[i]}-\hat{x}_{[i]}^T\Xi^2\hat{x}_{[i]}
       \end{aligned}
       \label{eq:Vdot8}
       \end{equation}
For (\ref{eq:Vdot8}) to be negative semi-definite, $\tilde{x}_{[i]}^T\varepsilon_i\tilde{x}_{[i]} > \hat{x}_{[i]}^T\Xi^2\hat{x}_{[i]}$. Ideally $\tilde{x}_{[i]}^T\varepsilon_i\tilde{x}_{[i]} \not> \hat{x}_{[i]}^T\Xi^2\hat{x}_{[i]}$. Hence, as expected it is difficult to guarantee GAS through local adaptive controllers. This problem is associated with the unmatched interconnections.
    
    \subsection{Proof of Lemma 2}\label{sec:7.3}
    The proof of this Lemma is an altered version of theorem 2 in \cite{Rajamani1998}, which provides sufficient conditions for ensuring asymptotic stability of an observer. The proof is performed in 3 parts.
    \begin{proof}
    
    \textbf{\textit{Part I:}} If $\min\limits_{\omega \in R^+}\sigma_{min}(\hat{A}_{m}-j\omega \mathbb{I}_{n_i\textrm{x}n_i}) > \Xi\sqrt{N_i}$ then there exists $\varepsilon_i > 0$ such that the matrix $\mathcal{H}_i = \left[ \begin{array}{cc}
    \hat{A}_{m} & N_i\mathbb{I}_{n_i\textrm{x}n_i}\\
    -(\Xi^2+ \varepsilon_i)\mathbb{I}_{n_i\textrm{x}n_i} & -\hat{A}_{m}^T
    \end{array} \right]$ is Hamiltonian, i.e. no eigenvalues on the imaginary axis.
    
    Since $\lim\limits_{\omega\rightarrow \infty}\sigma_{min}(\hat{A}_{m}-j\omega \mathbb{I}_{n_i\textrm{x}n_i}) = \infty $, there exists a finite $\omega \in R^+$ such that,
    \begin{equation}
     \min\limits_{\omega \in R^+}\sigma_{min}(\hat{A}_{m}-j\omega \mathbb{I}_{n_i\textrm{x}n_i}) = (\Xi\sqrt{N_i})_{min} 
     \end{equation} 
    that is, the minimum of all possible eigenvalues across all $\omega$ equals the minimum eigenvalues of $(\hat{A}_{m} - j\omega \mathbb{I}_{n_i\textrm{x}n_i}) $ at $\omega \in R^+$ which equals $(\Xi\sqrt{N_i})_{min}$. Hence, for all $\omega$, 
    \begin{equation}
    (\hat{A}_{m}-j\omega \mathbb{I}_{n_i\textrm{x}n_i})^*(A_{m_{i}}-j\omega \mathbb{I}_{n_i\textrm{x}n_i}) \geq (\Xi\sqrt{N_i})_{min}^2
    \end{equation}
    where, * denotes a Hermitian matrix, i.e. a transposed and complex conjugated matrix, or $A^* = \bar{A}^T$. Then, if $\min\limits_{\omega \in R^+}\sigma_{min}(A_m-j\omega \mathbb{I}_{n_i\textrm{x}n_i}) > \Xi\sqrt{N_i}$, $(\Xi\sqrt{N_i})_{min} > \Xi\sqrt{N_i}$. Therefore, we can choose $\varepsilon_i$ such that $(\Xi\sqrt{N_i})^2(\mathbb{I}_{n_i\textrm{x}n_i}+\varepsilon_i) < (\Xi\sqrt{N_i})_{min}^2$, in order to obtain,
    \begin{equation}
    \begin{aligned}
    (\hat{A}_{m}-j\omega \mathbb{I}_{n_i\textrm{x}n_i})^*(\hat{A}_{m}-j\omega \mathbb{I}_{n_i\textrm{x}n_i}) > (\Xi\sqrt{N_i})^2\mathbb{I}_{n_i\textrm{x}n_i} + (\Xi\sqrt{N_i})^2\varepsilon_i \mathbb{I}_{n_i\textrm{x}n_i}
    \end{aligned}
    \end{equation}
    
    Finally, part I can be proved by contradiction. Assuming that $\mathcal{H}_i$ has eigenvalues on the imaginary axis, where purely imaginary eigenvalues are represented as $s = j\omega$. The eigenvalues of $\mathcal{H}_i$ are given as,
    \begin{equation}
    \begin{aligned}
    det\left[ \begin{array}{cc}
    s\mathbb{I}_{n_i\textrm{x}n_i} - \hat{A}_{m} & -N_i\mathbb{I}_{n_i\textrm{x}n_i}\\
    (\Xi^2+ \varepsilon_i)\mathbb{I}_{n_i\textrm{x}n_i} & s\mathbb{I}_{n_i\textrm{x}n_i}+\hat{A}_{m}^T
    \end{array} \right]
    =0
    \\det[(s\mathbb{I}_{n_i\textrm{x}n_i} - \hat{A}_{m})(s\mathbb{I}_{n_i\textrm{x}n_i}+\hat{A}_{m}^T)+N_i\mathbb{I}_{n_i\textrm{x}n_i}(\Xi^2+ \varepsilon_i)]=0
    \end{aligned}
    \label{eqn:det}
    \end{equation}
    However, from (\ref{eqn:det}), $det[(s\mathbb{I}_{n_i\textrm{x}n_i} - \hat{A}_{m})(s\mathbb{I}_{n_i\textrm{x}n_i}+\hat{A}_{m}^T)+N_i\mathbb{I}_{n_i\textrm{x}n_i}(\Xi^2+ \varepsilon_i)]\not=0$. Therefore, $\mathcal{H}_i$ does not have imaginary axis eigenvalues, i.e. is $\mathcal{H}_i$ hyperbolic.
   \newline  
    \textbf{\textit{Part II:}} If $\mathcal{H}_i$ is hyperbolic, and $\hat{A}_{m}$ is Hurwitz, then there exists a symmetric positive-definite
    solution $P_{i}$ to the ARE,
    $$\hat{A}_{m}^TP_{i} +P_{i}\hat{A}_{m} + P_{i}N_iP_{i}+(\Xi^2+ \varepsilon_i)\mathbb{I}_{n_i\textrm{x}n_i} = 0$$
    \textit{Proof:}
    From $H_\infty$ control theory (chapter 7 of \cite{Francis1987}), if the matrix
    $$F := \left[ \begin{array}{cc}
    A & R\\
    Q & -A^T
    \end{array} \right]$$
    is hyperbolic, $R$ is either positive semi-definite or negative semi-definite, and ($A, R$) is stabilisable, then there exists a symmetric solution to the ARE,
    \begin{equation}
    A^TX + XA + XRX - Q = 0
    \end{equation}
    Furthermore, as $A$ is Hurwitz, $X$ is positive definite.
    \newline  
    \textbf{\textit{Part III:}} The existence of a positive-definite matrix $P_{i}$ that solves the ARE, ensures asymptotically stability. This can easily be proved by general Lyapunov theory of linear systems. 
    \end{proof}
    
    \subsection{Projection operator}\label{sec:7.4}
    Projection-based update laws for adaptation are commonly used in adaptive systems to bound estimates and prevent parameter drift. Properties of the projection operator are summarised following \cite{Yoo2010,L12010,Lavretsky2011}.
  \newline
    \textbf{Definition 1.} \textit{A set, defined as $\Omega \in \mathbb{R}^n$, is a convex set if for all values within this set i.e. $x \in \Omega$ and $y \in \Omega$, the following holds,}
    \begin{equation}
    \lambda x + (1-\lambda)y \in \Omega
    \end{equation}
    \textit{where, $0 < \lambda < 1$. That this,  all points connecting a straight line drawn between $x$ and $y$ remain within $\Omega$. }
    \newline
    \textbf{Definition 2.} \textit{A function $g:\mathbb{R}^n\rightarrow \mathbb{R}$, i.e. $g$ is a function on the set of real numbers of $n$ dimension into the set of real scalars, is convex if,}
    \begin{equation}
    g(\lambda x + (1-\lambda)y) \leq g(\lambda x) + g((1-\lambda)y)
    \end{equation}\textbf{Definition 3.}\textit{ Defining a convex set with a smooth boundary given by,}
  \begin{equation}
    \Omega_c \triangleq \{\theta \in \mathbb{R}^{n}\hspace{1mm}|\hspace{1mm} g(\theta)<\delta \}\hspace{3mm}
    ;\hspace{3mm} 0<\delta<1
    \end{equation}
    \textit{where, $g:\mathbb{R}^n\rightarrow\mathbb{R}$ is the convex function,}
    \begin{equation}
    g(\theta) \triangleq \frac{(\epsilon_0+1)\theta^T\theta-\theta_{max}^2}{\epsilon_0\theta_{max}^2}
    \label{eqn:g}
    \end{equation} 
    $\theta_{max}$ is the 1-norm bound of the parametric estimate vector $\theta$ which bounds the projection operation, and $\epsilon_0$ is the arbitrary projection tolerance bound which determines how close to the bound that scaling takes place. The projection operator is defined as,
    \begin{equation}
    \begin{aligned}
    Proj(\hat{\theta}_{[i]},-\tilde{x}_{[i]}^TP_{i}\bar{B}_i\bar{x}_{[i]})) \triangleq 
                   \begin{cases}
                  -\tilde{x}_{[i]}^TP_{i}\bar{B}_i\bar{x}_{[i]}, \hspace{72mm} $if$ \hspace{1.5mm} g(\theta) < 0 \\
                   -\tilde{x}_{[i]}^TP_{i}\bar{B}_i\bar{x}_{[i]}, \hspace{63.5mm} $if$ \hspace{1.5mm} g(\theta) > 0 \hspace{1.5mm} $, and$ \\
         \hspace{73mm} \hspace{1.5mm} \nabla g^T(-\tilde{x}_{[i]}^TP_{i}\bar{B}_i\bar{x}_{[i]}) \leq 0 \\
                   -\tilde{x}_{[i]}^TP_{i}\bar{B}_i\bar{x}_{[i]}-\frac{\nabla g}{||\nabla g||}(\frac{\nabla g}{||\nabla g||}.(-\tilde{x}^T_{[i]}P_{i}\bar{B}_i\bar{x}_{[i]})g(\theta),\hspace{10.5mm} $if$ \hspace{1.5mm} g(\theta) \geq 0 \hspace{1.5mm} $, and$ 
                   \\
                   \hspace{73.5mm} \nabla g^T(-\tilde{x}_{[i]}^TP_{i}\bar{B}_i\bar{x}_{[i]})\geq 0
                   \end{cases}
                   \label{eqn:projection}
                   \end{aligned}
                   \end{equation}
    From (\ref{eqn:g}), for the case when $g(\theta) \leq 1$,
    $$\theta^T\theta \leq \theta_{max}^2$$
    This shows that $\theta_{max}$ bounds the parameter estimate.
    For the case when $g(\theta) \geq 0$, $$\theta^T\theta  \geq \frac{\theta_{max}^2}{\epsilon_0+1}$$
    When $g(\theta) \geq 0$, $\theta$ becomes a scaled version of $\theta_{max}$; scaling begins at $\epsilon_0$.
    \vspace{3mm} \newline
    \textbf{Property 1.} \textit{The projection operator does not change the estimate, i.e. $\dot{\hat{\theta}}_{[i]} = 0$, if $g(\theta) \leq 0$ or if $g(\theta) > 0$ and $\nabla g^T(-\tilde{x}_{[i]}^T(t)P_{i}\bar{B}_i\bar{x}_{[i]}(t)) \leq 0$ i.e. a decreasing gradient. If $0 < g(\theta) < 1$ and $\nabla g^T(-\tilde{x}_{[i]}(t)^TP_{i}\bar{B}_i\bar{x}_{[i]}(t)) > 0$, i.e. the gradient of $g(\theta)$ is increasing, then the projection operator subtracts a vector normal to the boundary level set by $\theta_{max}$, as shown in (76).}
    \vspace{1mm} \newline
    For the Lyapunov function of (\ref{eqn:LyapVi}) to be at least negative the the following property is required.
    \vspace{3mm} \newline
    \textbf{Property 2.} \textit{Given the vector $\alpha \in \mathbb{R}^n$, $\theta_{[i]}^* \in \Omega_0$ is the uncertain parameter that is being estimated, and $\hat{\theta}_{[i]} \in \Omega_0$ is the estimate, then}
    \begin{equation}
    (\hat{\theta}_{[i]}-\theta_{[i]}^*)^T(Proj(\hat{\theta}_{[i]},\alpha)-\alpha) \leq 0
    \end{equation} \newline
    Indeed we have,
    \begin{equation}
    (\hat{\theta}_{[i]}-\theta_{[i]}^*)(\tilde{x}_{[i]}^TP_{i}\bar{B}_i\bar{x}_{[i]}+Proj(\dot{\hat{\theta}}_{[i]},\tilde{x}_{[i]}^TP_{i}\bar{B}_i\bar{x}_{[i]}(t)))
    \end{equation}
    which, using equation (76), properties 1 and 2 yields,
    \begin{equation}
    \begin{aligned}
    (\hat{\theta}_{[i]}-\theta_{[i]}^*)(\tilde{x}_{[i]}^TP_{i}\bar{B}_i\bar{x}_{[i]})+Proj(\dot{\hat{\theta}}_{[i]},\tilde{x}_{[i]}^TP_{i}\bar{B}_i\bar{x}_{[i]})) = \\
               \begin{cases}
               0, \hspace{65mm} $when$ \hspace{1.5mm} g(\theta) < 0 \hspace{1.5mm} $since$ \hspace{1.5mm} Proj(\dot{\hat{\theta}}_{[i]},\tilde{x}_{[i]}^TP_{i}\bar{B}_i\bar{x}_{[i]})) = - \tilde{x}_{[i]}^TP_{i}B\bar{x}_{[i]}) \\
               0, \hspace{65mm} $when$ \hspace{1.5mm} g(\theta) < 0 \hspace{1.5mm} $and$ \hspace{1.5mm} \nabla g^T(-\tilde{x}_{[i]}^TP_{i}\bar{B}_i\bar{x}_{[i]})\leq 0 \\
               \frac{\overbrace{(\hat{\theta}_{[i]}-\theta_{[i]}^*)\nabla g}^{\leq 0}\overbrace{\nabla g^T(-\tilde{x}_{[i]}^TP_{i}\bar{B}_i\bar{x}_{[i]})}^{> 0}\overbrace{g(\theta)}^{\geq 0}}{||\nabla g||^2}, \hspace{10mm} $when$ \hspace{1.5mm} g(\theta) \geq 0 \hspace{1.5mm} $and$ \\ \hspace{70mm} \nabla g^T(-\tilde{x}_{[i]}^TP_{i}\bar{B}_i\bar{x}_{[i]})> 0
               \label{eqn:projection2}
               \end{cases}
    \end{aligned}
    \end{equation}
    
    Ultimately, from (\ref{eqn:projection2}) it is seen that the projection operator continuously modifies the uniformly bounded adaptive law to maintain a negative semi-definite Lyapunov function derivative i.e. the left-hand-side term of (\ref{eqn:projection2}) is always less than or equal to zero in order to ensure the derivative of (\ref{eqn:LyapVi}) yields (\ref{eq:Vdot6}). 
   
    \subsection{Algorithm for Bisection Method}\label{sec:7.5}
    Summarising \cite{Byers1988,Aboky2002}, the bisection method is used to compute $\sigma_{min}(\hat{A}_{m}-j\omega \mathbb{I}_{n_i\textrm{x}n_i})$. The algorithm is:
               \begin{algorithm}
                   \caption{Bisection method for computing distance from stable to unstable poles}
                   \textbf{Require:} $\sigma_{min}(\hat{A}_{m}-j\omega \mathbb{I}_{n_i\textrm{x}n_i}) > \sqrt{N_i\Xi_{i}^2}$\\
                   \textbf{Input:} matrix $\hat{A}_m$, tolerance $\tau>0$\\ 
                   \textbf{Output:} lower-bound $\alpha$, upper-bound $\gamma$\\
                   \textbf{Initialisation} $\alpha = 0$, $\gamma = ||\hat{A}_m||_2, N = \frac{log_{2}\gamma}{2^\tau}$\\
                   \textbf{For} $i = 1:N$\\
                      $\sigma = \frac{\alpha+\gamma}{2}$\\
                      $\mathcal{H}_{i}(\sigma) = \left[ \begin{array}{cc}\\
                     \hat{A}_{m} & N_i\mathbb{I}_{n_i\textrm{x}n_i}\\
                     -\sigma^2\mathbb{I}_{n_i\textrm{x}n_i} & -\hat{A}_{m}^T\\
                     \end{array} \right]$\hspace{10mm}{define the Hamiltonian matrix}\\
                     \textbf{if} $\lambda_{min}(\mathcal{H}_{i}(\sigma)) \in \mathbb{R}^- \leq \tau $ \hspace{23mm}if minimum stable eigenvalue of $\mathcal{H}_{i}(\sigma)\leq\tau$\\
                     $\alpha = \sigma$
                    \textbf{else} \\
                     $\gamma = \sigma$\\
                      \textbf{end}
               \end{algorithm} 

 \subsection{Can typical local damping features improve global stability?}\label{sec:7.6}
               It is a common design feature in power electronics to increase the parasitic or equivalent series resistance (ESR) of each subsystem's output capacitor in order to improve damping. Section 
               6.3 of \cite{OKeeffe2018e} derives the state-space model which includes the capacitor ESR. The interconnection matrix is written as,
               
               \begin{equation}
               \hat{\zeta}_{[i]}(t) = \sum_{j\in \mathcal{N_i}}\left[ \begin{array}{ccc}
                                       0 & -\frac{(1-D_i)R_{ci}}{R_{ij}L_{ti}} & 0\\
                                       0 & \frac{1}{R_{ij}C_{t1}} & 0 \\
                                       0 & 0 & 0 \\
                                       \end{array} \right] \hat{x}_{[j]}(t)
               \label{eqn:tildeMod}
               \end{equation} 
               
             Though the ESR element is seen in the interconnection model, it does not influence it, either by introducing an additional eigenvalue or by nullifying the $\frac{1}{R_{ij}C_{ti}}$ term. Ultimately, despite being able to influence the dynamics and stability of the local subsystem, the ESR cannot help improve or relax global stability conditions.

     \bibliography{arXiv_Paper_TAC}

\end{document}